\documentclass[11pt,letterpaper]{amsart}
\usepackage[utf8]{inputenc}
\usepackage[T1]{fontenc}
\usepackage{amssymb}
\usepackage{graphicx}
\usepackage{url}
\usepackage{enumerate}
\usepackage{enumitem}
\usepackage{verbatim}
\usepackage{color}
\usepackage{tikz}
\usetikzlibrary{arrows.meta,bending}
\usetikzlibrary{decorations.pathmorphing}
\usetikzlibrary{decorations.markings}
\usepackage{array}
\usepackage{tikz-cd}
\usepackage{amsmath}
\usepackage{amsthm}
\usepackage[colorlinks=true,linkcolor=purple,citecolor=violet]{hyperref}
\usepackage{mathtools}
\usepackage{bbm}
\usepackage[capitalise, noabbrev]{cleveref}
\usepackage{fullpage}
\usepackage{lmodern}
\usepackage{mathrsfs}
\usepackage{esint}

\usetikzlibrary{positioning,arrows.meta}

\DeclareMathOperator{\bS}{\mathbb{S}}

\DeclareMathOperator{\ma}{\mathfrak{a}}
\DeclareMathOperator{\mpp}{\mathfrak{p}}

\DeclareMathOperator{\ttg}{\mathtt{g}}

\usepackage{lipsum}

\DeclareMathOperator{\sign}{sign}

\newtheorem{theorem}{Theorem}[section]
\newtheorem{lemma}[theorem]{Lemma}
\newtheorem{corollary}[theorem]{Corollary}
\newtheorem{proposition}[theorem]{Proposition}

\newtheorem{remark}[theorem]{Remark}

\theoremstyle{definition}
\newtheorem{definition}[theorem]{Definition}
\newtheorem{question}[theorem]{Question}
\newtheorem{example}[theorem]{Example}

\title{Poincar\'e-Einstein 4-manifolds with conformally K\"ahler geometry}
\author{Mingyang Li and Hongyi Liu}
\date{}

\newcommand{\Addresses}{{
  \bigskip
  \footnotesize

  \textsc{{Simons Center for Geometry and Physics, Stony Brook University, Stony Brook, NY 11794, USA}}\par\nopagebreak
  \textit{E-mail address}: \href{mailto:mingyang.li@scgp.stonybrook.edu}{\texttt{mingyang.li@scgp.stonybrook.edu}}

  \

  \textsc{{Department of Mathematics, Princeton University, Princeton,
    NJ 08544, USA.}}\par\nopagebreak
  \textit{E-mail address}: \href{mailto:hongyil@princeton.edu}{\texttt{hongyil@princeton.edu}}
}}

\begin{document}

\begin{abstract}
We study 4-dimensional Poincar\'e-Einstein manifolds whose conformal class contains a K\"ahler metric. Such Einstein metrics are non-K\"ahler and admit a Killing field extending to the conformal infinity, and the Einstein equation reduces to a Toda-type equation. When the Killing field integrates to an $\mathbb{S}^1$-action, we formulate a Dirichlet boundary value problem and establish existence and uniqueness theory. This construction provides a non-perturbative realization of infinite-dimensional families of new Poincar\'e-Einstein metrics whose conformal infinities are of non-positive Yamabe type.
\end{abstract}

\maketitle


\section{Introduction}

In this article, we study Poincar\'e--Einstein (PE) 4-manifolds that are conformally K\"ahler, i.e., whose conformal class contains a K\"ahler metric. Poincar\'e-Einstein manifolds, also known as conformally compact Einstein (CCE) or asymptotically hyperbolic Einstein (AHE) manifolds, have been extensively investigated due to their notable connections to the AdS/CFT correspondence in physics.

\begin{definition}\label{def:basic} 
    A \emph{Poincaré–Einstein} 4-manifold is an Einstein 4-manifold \((M^4, h)\) with $\operatorname{Ric}=-3h$, such that \(h\) is \emph{conformally compact} in the sense that \(M\) is the interior of a compact manifold \(\overline{M}\) with boundary \(\partial M\), and there exists a smooth defining function \(\rho\) on \(\overline{M}\), where \(\rho > 0\) on \(M\), \(\rho = 0\) on \(\partial M\), and \(d\rho\neq 0 \) on \(\partial M\), such that the conformally changed metric
        \[
            g = \rho^2 h
        \]
        extends smoothly to \(\overline{M}\).
We say $(\overline{M},g)$ is a \emph{conformal compactification}, and the conformal class $[g|_{\partial M}]$ is the \emph{conformal infinity}. 
\end{definition}

The definition implies that $h$ is complete, with $|d\rho|_{g}^{2}=1$ on $\partial M$,  and that sectional curvatures of $h$ converge uniformly to $-1$ near infinity,  by the conformal change formulas for scalar and Weyl curvatures.
The study of such metrics goes back to the foundational works of Fefferman--Graham \cite{FeffermanGraham1985,fefferman_graham_2012} on conformal invariants and has since been advanced by Graham--Lee, Anderson, Biquard, et al.\ \cite{graham_lee_1991,Anderson2001,Biquard2000}. Their dependence on the boundary conformal structure makes Poincar\'e--Einstein manifolds a natural setting for geometric boundary value problems. A fundamental example of a Poincar\'e--Einstein metric is the standard hyperbolic metric on the 4-ball $B^4$, 
$$
h=\frac{4}{(1-|x|^2)^2}(dx_1^2+\cdots+dx_{4}^2), \qquad 
\rho=\tfrac{1-|x|^2}{2}.
$$
Other important examples include the AdS–Schwarzschild, AdS–Kerr, Page–Pope \cite{PagePope1987}, Pedersen \cite{Pedersen1986}, Calderbank–Singer \cite{CalderbankSinger2004}, and Plebański–Demiański \cite{PlebanskiDemianski1976} families (further studied by Alvarado–Ozuch–Santiago \cite{Ozuch2023}), all of which are conformally Kähler, as well as the anti-self-dual examples of Hitchin \cite{Hitchin1995Twistor}. More recently, Wang \cite{wang2025computerassistedconstructionsu2invariantnegative} constructed new examples that are not self-dual and cohomogeneity-one.

A central theme in the study of Poincar\'e--Einstein manifolds is the \textbf{fill-in problem}:

\begin{question}\label{ques: fundamental CCE}
Given a conformal class $[\hat g]$ on $\partial M$, does there exist a Poincar\'e-Einstein metric $h$ on a 4-manifold $M$ whose conformal infinity is $[\hat{g}]$?
\end{question}

It was shown by Fefferman-Graham that if $(M^4,h)$ is a Poincar\'e--Einstein manifold with conformal infinity $[\hat g]$, then for any choice of representative metric $\hat g\in[\hat g]$ there exists a \emph{geodesic defining function} $r$ such that the compactified metric $g'=r^2h$ satisfies
$|\nabla_{g'} r|_{g'} = 1  $ near $\partial M$ and $g'|_{\partial M}=\hat g$.
With respect to this defining function, the compactified metric $g' = r^2 h$ admits a formal expansion
\begin{equation*}
g' = dr^2 + \hat g + r^2 g_{(2)} + r^3 g_{(3)} + \cdots,
\end{equation*}
where the coefficient $g_{(2)}$ is locally determined by the curvature of $\hat g$, while $g_{(3)}$ is a non-local term, depending globally on $h$ and $\hat{g}$. We call such $(\overline{M},g')$ a \emph{geodesic compactification}. Biquard \cite{biquard2008continuationunique} proved that $h$ is determined up to local isometry by the pair $(\hat g,g_{(3)})$.  

For Question~\ref{ques: fundamental CCE}, there are local existence results, valid only near infinity. LeBrun~\cite{LeBrun1985} and Fefferman–Graham~\cite{fefferman_graham_2012} established such results for real-analytic conformal boundaries, later extended to smooth boundaries by Gursky–Székelyhidi~\cite{gursky_szekelyhidi_2020}.

By contrast, the \emph{global} fill-in problem is far more subtle, since one must decide which local solutions defined in a collar neighborhood extend smoothly to the whole manifold. The perturbative approach was initiated by Graham–Lee \cite{graham_lee_1991} and later extended in greater generality by Lee \cite{lee_2006}. Anderson \cite{Anderson2003_AH_Einstein} revealed non-uniqueness and degeneration phenomena, such as cusp formation. Through twistor theory, Biquard \cite{biquard2002metriques} classified all local deformations of self-dual Poincaré–Einstein metrics near the hyperbolic 4–ball, thereby confirming LeBrun’s positive frequency conjecture \cite{lebrun1991complete}. More recently, Chang–Ge, Chang–Ge–Qing, and Chang–Ge–Jin–Qing \cite{chang_ge_2018,chang_ge_qing_2020,chang_ge_jin_2024} established compactness results and perturbative existence theorems. In their latest work, Chang–Ge \cite{chang2025problemfillingpoincareeinsteinmetric} proved that if a metric $\hat g$ lies within a definite-size neighborhood of either the round metric on $\mathbb{S}^3$ or the product metric on $\mathbb{S}^1 \times \mathbb{S}^2$, then $\hat g$ admits a Poincar\'e–Einstein fill-in with underlying manifold given by $B^4$ or $S^1 \times B^3$, respectively.

However, little existence theory is known at large scales. The main goal of this article is to construct Poincaré–Einstein metrics whose conformal infinities are far from known examples. The main result of this paper is
\begin{theorem}\label{thm:main conclusion}
Given a closed, oriented, smooth surface $(\Sigma_{\ttg},g^\natural)$ of genus $\ttg \geq 1$, area $\ma$, and Euler characteristic $\chi = 2-2\ttg$, suppose $k \in \mathbb{R}\cup\{\pm\infty\}$ and $deg \in \mathbb{Z}$ are such that the tuple $(deg,\chi,k,\ma,\mpp)$ is admissible.  
Then there exists a Poincar\'e--Einstein $(M,h)$ with conformal infinity $[g^\flat]$ implicitly determined by $g^\natural$ and the tuple, where 
\begin{itemize}
    \item $M$ is diffeomorphic to the complex line bundle of degree $deg$ over $\Sigma_{\ttg}$;
    \item $(\partial M,g^\flat)$ carries a free isometric $\mathbb{S}^1$-action with quotient $(\Sigma_{\ttg},g^\natural)$ and constant orbit length $\mpp$.
\end{itemize}
\end{theorem}

The tuple $(deg,\chi,k,\ma,\mpp)$ is called \emph{admissible} if it satisfies the algebraic condition given in Definition~\ref{def:admissible five tuple}. In particular, for any fixed $deg$ and $\chi$, there always exist parameters $k,\ma,\mpp$ making the tuple admissible. The number $k$ here has the geometric meaning that for the PE fill-in $(M,h)$, $|k|\cdot (2\sqrt{6}\max|\mathscr{W}_h^+|_h)^{1/3}=\frac{1}{2}$ where $\mathscr{W}_h^+$ refers to the self-dual Weyl curvature of $(M,h)$, and the fill-in is anti-self-dual precisely when $k=\pm\infty$. The boundary metric $g^\flat$ is then determined implicitly by this tuple together with $(\Sigma_{\ttg},g^\natural)$. In the special case $k=\pm\infty$, the admissible condition means $deg<-\chi/2$, $\mpp=\pi$, and $\ma=-\pi(deg+\chi/2)$.

The PE manifolds arising from Theorem \ref{thm:main conclusion} are all \textit{conformally K\"ahler}. Our motivation comes from a result of LeBrun \cite{lebrun20}, who showed that if a closed Bach-flat Kähler surface is not Kähler–Einstein and has scalar curvature $s_g<0$ somewhere, then the zero set of $s_g$ divides the manifold into two regions, each carrying a Poincaré–Einstein metric with $s_g$ as a defining function. This connection between Kähler geometry and Poincaré–Einstein metrics motivates our focus on the conformally Kähler setting.

To place this in a broader context, we recall Derdziński’s characterization of Einstein metrics conformal to Kähler metrics. For an oriented Einstein $4$-manifold $(M,h)$, Derdziński \cite{derdzinski} showed that if the self-dual Weyl curvature $\mathscr{W}_h^+$ has a nonzero repeated eigenvalue everywhere (Type II), then $h$ is locally conformal to a Kähler metric $g$, which up to scaling by constants the conformal factor is canonically determined by $|\mathscr{W}_h^+|_h$. If moreover $h$ is not Kähler–Einstein (in which case the conformal change is trivial), the Kähler metric $g$ is actually an extremal Kähler metric with non-constant scalar curvature, and therefore admits a Killing field $\mathcal{K}$. As for anti-self-dual Einstein $4$-manifolds (that is, $\mathscr{W}^+_h\equiv0$) with non-zero Einstein constant, they need not admit a Killing field, but those that do are precisely the ones locally conformal to Kähler metrics. Conversely, by \cite{derdzinski}, any Einstein metric with non-zero Einstein constant that is conformal to a Kähler metric must be either anti-self-dual everywhere or of Type II. In both cases, the Einstein metric automatically admits a Killing field $\mathcal{K}$ that is canonically associated to the Kähler metric $g$ (see Section \ref{subsec:Type I with Kahler} and \ref{subsec:Type II with Kahler}).

These Killing fields play a crucial role in the conformally Kähler Poincaré–Einstein setting. On such a manifold $(M,h)$, a Killing field $\mathcal{K}$ extends to any conformal compactification as a conformal Killing field tangent to the boundary.  The Kähler conformal change $g$ yields a compactification precisely when $\mathcal{K}$ has no zeros on $\partial M$. Moreover, $\mathcal{K}$ is simultaneously a Killing field of both $h$ and $g$, and in the Type~II case it is always nowhere vanishing on the boundary.

We now focus on the case where $\mathcal{K}|_{\partial M}$ is periodic and integrates to a free $\mathbb{S}^1$-action on $\partial M$, so that $\partial M$ is a principal $\mathbb{S}^1$-bundle over a compact oriented smooth surface $\Sigma$. In this situation we say that $(M,h)$ has \emph{regular conformally Kähler geometry}. In this case, the $\mathbb{S}^1$-action is Hamiltonian with respect to the Kähler form, and one can perform Kähler reduction away from critical values of the moment map $\xi$ of $\mathcal{K}$ to write the Kähler metric $g$ in the form 
$$
g = W d\xi^2 + W^{-1}\eta^2 + W e^w g_{\Sigma},
$$
where $W,w$ functions of $\xi$ and the base $\Sigma$, and $\eta$ a multiple of a connection $1$-form, dual to $\mathcal{K}$. Here, $We^w g_\Sigma$ is the K\"ahler metric on the K\"ahler reduction $\Sigma$ and $g_\Sigma$ is chosen to be the metric with constant curvature $K_\Sigma\in\{0,\pm1\}$ on the Riemann surface $\Sigma$, and if $K_\Sigma=0$ the area of $g_{\Sigma}$ is normalized to be $4\pi^2$. By scaling $g$ properly, we can arrange  that $$g=\xi^2h, \quad\mathcal{K}=J\nabla_g\xi,$$ and the moment map $\xi$ to have range $[0,\frac{1}{2}]$ on $\overline{M}$, where $\partial M=\{\xi=0\}$. The choice of $[0,\frac{1}{2}]$ is natural as this is the range for the hyperbolic ball $B^4$ (see Example \ref{example:hyperbolic space}). The moment map $\xi$ is a Morse-Bott function, as previously observed by LeBrun \cite{lebrun20}. One then derives that $W,w$ satisfy the following PDE system:
\begin{equation}\label{eq:introduction-equation-w}
(e^{w})_{\xi\xi}+\Delta_{\Sigma}w-2K_{\Sigma}
=-\xi e^{w}\frac{12-6\xi \partial_{\xi}w}{12k^3+\xi^3},
\end{equation}
\begin{equation}\label{eq:introduction-equation-W}
W=\frac{12-6\xi w_{\xi}}{12+\xi^3/k^3},
\end{equation}
where $k$ is a constant arising from the previous scaling and satisfies
$$|k|\cdot (2\sqrt{6}\max|\mathscr{W}_h^+|_h)^{1/3}=\frac{1}{2},$$ 
measuring how far the Einstein metric is from being anti-self-dual (see Section \ref{subsec:canonical conformal change}). The metric is anti-self-dual precisely when $k=\pm\infty$. This reduction is essentially due to LeBrun \cite{lebrun}, and Tod \cite{tod95}.

Next we move to the case that the topology of $M$ is relatively simple, where the Morse-Bott function $\xi$ only has an isolated critical point or a critical surface, at which the maximum of $\xi$ is achieved. The underlying manifold $M$ is diffeomorphic to $B^4$ or a complex line bundle over a closed Riemann surface $\Sigma$.
The following construction arises from explicit solutions to \eqref{eq:introduction-equation-w}–\eqref{eq:introduction-equation-W} which depends only on $\xi$:

\begin{theorem}\label{thm:introduction cohomogenity one examples}
On $B^4$, or on any $M$ diffeomorphic to the total space of a complex line bundle over a closed Riemann surface $\Sigma$, there exists a one-parameter family of cohomogeneity-one Poincar\'e--Einstein metrics with regular conformally Kähler geometry. 
\end{theorem}

Moreover, this family contains an anti-self-dual metric if and only if the degree of the complex line bundle satisfies $deg < \min\{-\chi(\Sigma), -\chi(\Sigma)/2\}$. In particular, this recovers the constructions of Page–Pope \cite{PagePope1987}, Oliveira–Sena Dias \cite{OliveiraSenaDias2025}, and Pedersen \cite{Pedersen1986}.

In general, supposing that $\ttg(\Sigma)\geq 1$ and the tuple $(deg,\chi,k,\ma,\mpp)$ is admissible in the sense of Definition \ref{def:admissible five tuple},  we prove that the following Dirichlet boundary value problem admits unique solution:
\begin{theorem}\label{thm:introduction existence and uniqueness} 
    For any $\varphi\in C^{2,\alpha}(\Sigma)$ with $\int_{\Sigma} e^\varphi d\mathrm{vol}_{{\Sigma}}=\mathfrak{a}$, there exists a unique solution $w$ of 
    \begin{equation}
        \left\{
            \begin{aligned}
                &(e^w)_{\xi\xi}+\Delta_\Sigma w-2K_\Sigma=-\xi e^w\frac{12-6\xi \partial_\xi w}{12k^3+\xi^3},\\
                &w|_{\xi=0}=\varphi,
            \end{aligned}
        \right.
    \end{equation}
    such that $w-\log(\frac{1}{2}-\xi)\in C^{2,\alpha}([0,\frac{1}{2}]\times \Sigma)$. Moreover, $w-\log(\frac{1}{2}-\xi)\in C^{\infty}((0,\frac{1}{2}]\times \Sigma)$.
\end{theorem}

The elliptic equation above is degenerate at the boundary $\xi=\frac{1}{2}$. But the degeneration is mild and we shall overcome this by lifting the equation to an equation on $\mathbb{R}^4\times\Sigma$ by setting $\xi=\frac{1}{2}-\frac{1}{4}|\mathbf{z}|^2$ with $\mathbf{z}$ parametrizes $\mathbb{R}^4$, where the boundary $\{\frac{1}{2}\}\times\Sigma$ becomes interior part of the domain. The tuple $(deg,\chi,k,\ma,\mpp)$ being admissible, together with a regularity result for Einstein metrics, guarantee that solutions $w$ with smooth boundary value correspond to smooth Poincar\'e-Einstein fill-ins, with the underlying manifold diffeomorphic to  complex line bundles over $\Sigma$. In particular, the corresponding function $W$ defined by \eqref{eq:introduction-equation-W} is positive on $[0,\tfrac{1}{2})$ and blows up at $\xi=\frac{1}{2}$, which corresponds to the fixed surface of the $\mathbb{S}^1$-action. From such a solution $w$ and the admissible tuple, we consider the $\mathbb{S}^1$–bundle $P\to[0,\tfrac12)\times\Sigma$ of degree $deg$, together with a $1$–form $\eta$ such that $\tfrac{2\pi}{\mathfrak{p}}\eta$ is a connection $1$–form, where the curvature of the bundle is determined by
$$
d\eta=(W e^{w})_{\xi}\,d\!\operatorname{vol}_{\Sigma}
+(J_{\Sigma}\cdot d_{\Sigma}W)\wedge d\xi.
$$
Then by a regularity result, the completion of the metric 
$$
g=Wd\xi^2+W^{-1}\eta^2+We^w g_{\Sigma}
$$
is a smooth Kähler metric on $M$, where diffeomorphically 
$M \to \Sigma$ is the total space of the complex line bundle of degree $deg$, and $(M,\xi^{-2}g)$ is the Poincar\'e--Einstein fill-in. This indicates Theorem \ref{thm:main conclusion}. Finally, we also show that 
\begin{theorem}\label{thm:introduction uniqueness of poincare einstein}
    Every Poincar\'e--Einstein manifold with regular conformally K\"ahler geometry on the total space of a complex line bundle over an oriented surface of genus greater than zero arises from the construction correpsonds to an admissible tuple and a solution $w$ in Theorem \ref{thm:introduction existence and uniqueness}. 
\end{theorem}

\medskip
{\bf Notations and Conventions}

\begin{itemize}
    \item We abbreviate Poincar\'e-Einstein 4-manifolds (see Definition \ref{def:basic}) as PE manifolds. 
    \item We denote a regular conformally K\"ahler PE (see Definition \ref{def:PE conformally Kahler} and \ref{def:regular conformally Kahler}) with its canonical conformal change (see Definition \ref{def:canonical infinity}) as a triple $(M,h,g)$. 
    \item Denote the genus of an oriented smooth surface by $\ttg$.
    \item All Einstein manifolds are normalized to satisfy $\operatorname{Ric}=-3h$.
    \item Given a Riemann surface $\Sigma$, we denote by $g_{\Sigma}$ the metric with Gauss curvature $K_\Sigma\in\{0,\pm1\}$, and if $K_{\Sigma}=0$ the area of $g_{\Sigma}$ is normalized to be $4\pi^2$.
\end{itemize}

\medskip
{\bf Acknowledgment}
Both authors are grateful to Song Sun and Ruobing Zhang for many helpful discussions. The first author also thanks IASM at Zhejiang University for its hospitality during his visits in 2024 and Spring 2025, during which part of this work was carried out. The second author is indebted to Sun-Yung Alice Chang for many helpful conversations on Poincaré–Einstein manifolds, and is also grateful to Paul Yang and Yuxin Ge for related discussions.

\section{Poincar\'e-Einstein metrics with conformally K\"ahler geometry}

On an oriented Riemannian 4-manifold \( (M^4,h) \), the bundle of 2-forms decomposes into self-dual and anti-self-dual components:
\[
\Lambda^2 T^*M = \Lambda^+ \oplus \Lambda^-.
\]
With respect to this splitting, the Riemann curvature operator  
\[
\mathrm{Rm} : \Gamma(\Lambda^+ \oplus \Lambda^-) \longrightarrow \Gamma(\Lambda^+ \oplus \Lambda^-)
\]
admits the following decomposition:
\[
\mathrm{Rm} =
\begin{pmatrix}
\mathscr{W}^+ + \tfrac{s_h}{12}\,\mathrm{Id} & \mathring{\mathrm{Ric}} \\[6pt]
\mathring{\mathrm{Ric}} & \mathscr{W}^- + \tfrac{s_h}{12}\,\mathrm{Id}
\end{pmatrix},
\]
where \( \mathscr{W}^+ : \Gamma(\Lambda^+) \to \Gamma(\Lambda^+) \) and \( \mathscr{W}^- : \Gamma(\Lambda^-) \to \Gamma(\Lambda^-) \) denote the self-dual and anti-self-dual components of the Weyl tensor \( \mathscr{W} \) respectively, $s_h$ is the scalar curvature, and \( \mathring{\mathrm{Ric}} \) is the trace-free part of the Ricci tensor. Viewing \( \mathscr{W}^+ \) as a symmetric, trace-free endomorphism, we order its eigenvalues as \( \lambda_1(\mathscr{W}^+) \geq \lambda_2(\mathscr{W}^+) \geq \lambda_3(\mathscr{W}^+) \), satisfying the identity \( \lambda_1(\mathscr{W}^+) + \lambda_2(\mathscr{W}^+) + \lambda_3(\mathscr{W}^+) = 0 \). Denote $Spec(\mathscr{W}^+)=\{\lambda_1(\mathscr{W}^+), \lambda_2(\mathscr{W}^+), \lambda_3(\mathscr{W}^+)\}$.

An Einstein metric $h$ belongs to exactly one of the following categories, as established by Derdzi\'nski \cite{derdzinski}:
\begin{itemize}
	\item (Type I, or ASD)  if $\mathscr{W}^+=0$ everywhere.
	\item (Type II) $\mathscr{W}^+$ has two distinct eigen-values everywhere. 
	\item (Type III) $\mathscr{W}^+$ generically has three distinct eigen-values.
\end{itemize}

\begin{lemma}[\cite{derdzinski}]\label{lem:Derdzinski}
    If a connected 4-manifold $(M^4,h)$ is Einstein, and $\#Spec(\mathscr{W_h}^+)\leq 2$ everywhere, then $h$ is either anti-self-dual  everywhere, or $\#Spec(\mathscr{W_h}^+)= 2$ everywhere.
\end{lemma}

For any K\"ahler metric $g$ with K\"ahler form $\omega$, its self-dual Weyl curvature $\mathscr{W}_g^+$ is of the form with respect to the decomposition $\Lambda^+=\mathbb{R}\omega\oplus \text{Re} \ \Lambda^{2.0}$  
\begin{equation}\label{eq:self_dual_weyl_curvature_of_a_Kahler_metric}
\mathscr{W}^+_g=\left(\begin{matrix}
        s_g/6 & & \\
        & -s_g/12 & \\
        & & -s_g/12
    \end{matrix}\right)
    \end{equation}
Thus, if an Einstein metric \( h \) is locally conformal to a Kähler metric, then \( h \) is either of Type I or Type II. Conversely, if \( h \) is of Type II, then either \( \det(\mathscr{W}^+_h) > 0 \) everywhere, or \( \det(\mathscr{W}^+_h) < 0 \) everywhere, and the conformal changed metric $g := (2\sqrt{6}|\mathscr{W}^+_h|_h)^{2/3} h$ locally is an extremal K\"ahler metric whose K\"ahler form corresponds to an eigen-2-form associated with the simple eigenvalue of \( \mathscr{W}^+_h \), as shown in \cite{derdzinski,lebrun20}. Passing to a double cover if necessary so that there is a global eigen-2-form associated to the simple eigenvalue, $(M,h)$ is globally conformally K\"ahler with $g$ being extremal K\"ahler. We will explain Type II Poincar\'e-Einstein metrics in details in Section \ref{subsec:Type II with Kahler}. However, a Type I metric is not necessarily locally conformal to a K\"ahler metric. In fact, every Type I metric that is locally conformal to a K\"ahler metric admits a continuous symmetry, as we shall explain in Section~\ref{subsec:Type I with Kahler}.

The diagram below illustrates the possible degenerations for sequences of Einstein metrics. 
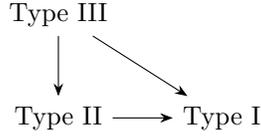
\begin{center}
\begin{tikzpicture}[>=Stealth, node distance=0.8cm, font=\small]
  \node (III) {Type III};
  \node (II) [below=of III] {Type II};
  \node (I)  [right=of II] {Type I};

  \draw[->] (III) -- (II);
  \draw[->] (II) -- (I);
  \draw[->] (III) -- (I);
\end{tikzpicture}
\end{center}
For Type I PEs, by the work of Biquard-Rollin \cite{BiquardRollin2009} we know the deformation problem for conformal infinity is unobstructed, so it is easy to construct sequence of Type III PEs converging to a Type I PE.

\ 

We now focus on Poincar\'e--Einstein $4$-manifolds $(M^4,h)$ that is conformal to a K\"ahler metric $g$.  
By Lemma~\ref{lem:Derdzinski} and \eqref{eq:self_dual_weyl_curvature_of_a_Kahler_metric}, if the scalar curvature of $g$ vanishes at one point, then it vanishes identically. Hence:

\begin{lemma}
The scalar curvature $s_g$ satisfies one of the following: $s_g>0$ on $M$, $s_g<0$ on $M$, or $s_g\equiv 0$ on $M$.
\end{lemma}

\subsection{Type I with conformally K\"ahler geometry}
\label{subsec:Type I with Kahler}

By works of Tod and Przanowski \cite{tod95,tod06,przanowski}, we have the following local result. There is also an explanation based on twistor theory due to LeBrun. See also \cite{calderbank,joefine}.

\begin{proposition}\label{prop:Type I with conformal Kahler}
  An ASD Einstein metric $h$ with non-zero scalar curvature and a non-trivial Killing field $\mathcal{K}$ is locally conformal to a K\"ahler scalar-flat metric $g$ on an open dense set. Conversely, an ASD Einstein metric $h$ with non-zero scalar curvature that is conformally K\"ahler admits a non-trivial Killing field $\mathcal{K}$.
\end{proposition}
\begin{proof}
  Given such an ASD Einstein metric $h$ and a Killing field $\mathcal{K}$, set $\theta_{\mathcal{K}}$ to be the dual one-form of $\mathcal{K}$. Take the self-dual part $(d\theta_{\mathcal{K}})^+$ and write $(d\theta_{\mathcal{K}})^+(\cdot,\cdot)=\zeta^{-1} g(\cdot ,J\cdot)$ with a non-negative function $\zeta$ and $J^2=-Id$,  then $J$ as an almost-complex structure is integrable, and the conformal metric $g:=\zeta^{2}h$ is K\"ahler scalar-flat with the complex structure $J$. Note that this only makes sense over regions where $(d\theta_{\mathcal{K}})^+\neq 0 $. Conversely, knowing such an ASD Einstein metric $h$ and a conformal K\"ahler metric $g=\zeta^2h$ whose K\"ahler form is $\omega$, define $\theta_\mathcal{K}:=\delta_h (\zeta^{-1}\omega)$ as the divergence of $\zeta^{-1}\omega$ under $h$, then the dual vector field $\mathcal{K}$ of $\theta_{\mathcal{K}}$ is a Killing field. See \cite{calderbank} for details.
\end{proof}

Note that the Killing field $\mathcal{K}$ and the conformally K\"ahler geometry $g=\zeta^2h$ are canonically related. Here, $\zeta$ up to scaling is the moment map of the Killing field $\mathcal{K}$ under the K\"ahler metric $g$. In particular, by scaling the Killing field $\mathcal{K}$ properly, we can always arrange $\mathcal{K}=J\nabla_g\zeta$. Therefore, an ASD Einstein metric $h$ with a conformally K\"ahler geometry $g$ can be locally written down using the LeBrun's ansatz for K\"ahler scalar-flat metric with $\mathbb{R}$-symmetry \cite{lebrun}, which is obtained by considering the K\"ahler reduction with respect to $\mathcal{K}=J\nabla_g\zeta$
\begin{equation}\label{eq:Type I metric}
  g=\zeta^2 h=W(d\zeta^2+e^v(dx^2+dy^2))+W^{-1}\eta^2,
\end{equation}
\begin{equation}\label{eq:Type I Toda}
  (e^v)_{\xi\xi}+v_{xx}+v_{yy}=0,
\end{equation}
\begin{equation}\label{eq:Type I W}
  W=1-\frac{1}{2} \zeta v_\zeta,
\end{equation}
\begin{equation}\label{eq:Type I d eta}
  d\eta=(We^v)_\zeta dxdy+W_xdyd\zeta +W_yd\zeta dx.
\end{equation}
The complex structure $J$ is given by $Jdx=dy$, $Jd \zeta=W^{-1}\eta$. Note that in LeBrun's ansatz in \cite{lebrun}, we do not have \eqref{eq:Type I W}, and instead one has $(We^v)_{\zeta\zeta} + v_{xx} + v_{yy} = 0$, which is implied by \eqref{eq:Type I Toda} and \eqref{eq:Type I W}. Only when the K\"ahler scalar-flat metric $g$ is conformal to an Einstein metric we have \eqref{eq:Type I W}, which follows from the conformal change for scalar curvature. 

The following example demonstrates that, even for the same ASD Poincar\'e–Einstein metric, different Killing fields can induce distinct conformally K\"ahler geometries.

\begin{example}[Hyperbolic 4-ball]\label{example:hyperbolic space}
Consider the standard hyperbolic metric on $B^4$
    $$h=\frac{4}{(1-r^2)^2}(dr^2+r^2g_{S^3})=\frac{4}{(1-r^2)^2}(dx_1^2+dx_2^2+dx_3^2+dx_4^2).$$ 
One can apply different Killing fields to obtain different conformally K\"ahler geometries. Set the orientation to be given by $dx_1 \wedge dx_2 \wedge dx_3 \wedge dx_4$.

We first consider $\mathcal{K} = x_1 \partial_{x_2} - x_2 \partial_{x_1}$.  Denote 
\begin{align*}
  \omega_1 &= \frac{4}{(1 - r^2)^2}(dx_1 \wedge dx_2 + dx_3 \wedge dx_4),\\
  \omega_2 &= \frac{4}{(1 - r^2)^2}(dx_1 \wedge dx_3 - dx_2 \wedge dx_4),\\
  \omega_3 &= \frac{4}{(1 - r^2)^2}(dx_1 \wedge dx_4 + dx_2 \wedge dx_3),
\end{align*}
so that $\omega_1, \omega_2, \omega_3$ form an orthogonal frame of the bundle of self-dual $2$-forms, each of norm $\sqrt{2}$. Then $\theta_{\mathcal{K}}=\frac{4}{(1-r^2)^2}\,(x_1dx_2-x_2dx_1)$, and 
$$(d\theta_{\mathcal{K}})^+=\frac{2}{1-r^2}\Bigl((x_1^2+x_2^2)\omega_1+(x_1x_4+x_2x_3)\omega_2+(x_2x_4-x_1x_3)\omega_3\Bigr)+\omega_1.$$
Hence 
$$\zeta^{-1}=|(d\theta_{\mathcal{K}})^+|_{h}=\frac{2}{1-r^2}\sqrt{\,2(x_1^2+x_2^2)+2\!\left(\tfrac{1-r^2}{2}\right)^{\!2}}.$$
The corresponding K\"ahler metric is 
$$g=\zeta^2h=\frac{1}{\,2(x_1^2+x_2^2)+2\!\left(\tfrac{1-r^2}{2}\right)^{\!2}}\,(dx_1^2+dx_2^2+dx_3^2+dx_4^2).$$
An easier way to understand the K\"ahler geometry here is to transform to the upper-half-space model $h=\frac{1}{w^2}(dx^2+dy^2+dz^2+dw^2)$ via the isometry 
$$x=\frac{2x_1}{Q},\quad y=\frac{2x_2}{Q},\quad z=\frac{2x_3}{Q},\quad w=\frac{1-r^2}{Q}, \quad \text{where}\quad Q=x_1^2+x_2^2+x_3^2+(x_4-1)^2.$$
Then the Killing field $\mathcal{K}=x\partial_y-y\partial_x$, and it is easy to check 
$g=\frac{1}{x^2+y^2+w^2}(dx^2+dy^2+dz^2+dw^2),$
which is just a Poincar\'e cusp $g_{S^2}+\tfrac{1}{r^2}(dz^2+dr^2)$ with $r^2=x^2+y^2+w^2$.

 Next, we take $\mathcal{K} = x_1\partial_{x_2} - x_2\partial_{x_1} + x_3\partial_{x_4} - x_4\partial_{x_3}$, the rotation in the Hopf fibration. A direct calculation gives
$$(d\theta_{\mathcal{K}})^+ = \frac{2}{1 - r^2}\,\omega_1.$$
Hence, $\zeta = \tfrac{1 - r^2}{2\sqrt{2}}$, and 
$g = \zeta^2 h = \tfrac{1}{2}(dr^2 + r^2 g_{S^3})$ with $ r < 1.$
This provides a flat conformal compactification of $(B^4,h)$, with conformal infinity $g|_{\{1\}\times S^3}=\tfrac{1}{2}g_{S^3}$.

More complicated examples can be obtained by considering the general rotational Killing field 
\[
\mathcal{K}=a(x_1\partial_{x_2}-x_2\partial_{x_1})+b(x_3\partial_{x_4}-x_4\partial_{x_3}).
\]
Here, $a$ and $b$ can even be taken as irrational numbers. Write $z_1=x_1+ix_2$ and $z_2=x_3+ix_4$, then
\begin{align*}
    \zeta^{-1}
    &=\frac{2}{1-r^2}\left|\Bigl(\tfrac{1}{2}(a-b)(|z_1|^2-|z_2|^2)+\tfrac{1}{2}(a+b)\Bigr)\omega_1
      +(a-b)\,\mathrm{Im}(z_1z_2)\,\omega_2-(a-b)\,\mathrm{Re}(z_1z_2)\,\omega_3\right|_h \\[4pt]
    &=\frac{2\sqrt{2}}{1-r^2}\left(\Bigl(\tfrac{1}{2}(a-b)(|z_1|^2-|z_2|^2)+\tfrac{1}{2}(a+b)\Bigr)^{2}
      +(a-b)^2|z_1z_2|^2\right)^{1/2}\\[4pt]
    &=\frac{2\sqrt{2}}{1-r^2}\left(\tfrac{1}{4}(a-b)^2(1-r^2)^2+ab(1-r^2)+a^2|z_1|^2+b^2|z_2|^2\right)^{1/2}.
\end{align*}
When neither $a$ nor $b$ is zero, the metric $g=\zeta^2h$ again provides a conformal compactification for $(B^{4},h)$:
$$g=\tfrac{1}{2}\Bigl(\tfrac{1}{4}(a-b)^2(1-r^2)^2+ab(1-r^2)+a^2|z_1|^2+b^2|z_2|^2\Bigr)^{-1}(dr^2+r^2g_{S^3}).$$
Here, $S^3$ is understood as the round sphere $\{|z_1|^2+|z_2|^2=1\}\subset\mathbb{C}^2$.

We may also consider the Killing fields
$
\mathcal{K}_i = \frac{1+r^2}{2}\partial_{x_i}-x_i \sum_{j=1}^4 x_j \partial_{x_j},\ i=1,2,3,4.
$
Let $e_i$ denote the point in $\mathbb{R}^4$ whose $i$-th coordinate is $1$ and all other coordinates are zero. For each $i$, the flow lines of $\mathcal{K}_i$ are arcs connecting $e_i$ and $-e_i$, and $\mathcal{K}_i$ is tangent to $\partial B^4$. The only two zeros of $\mathcal{K}_i$ on $\overline{B^4}$ are $\pm e_i$, and they are simple zeros.

We may also take the Killing field $\mathcal{K}$ in the ball model, which corresponds to a translation in the horizontal direction in the upper half-space model. The vector field $\mathcal{K}$ extends to $\overline{B^4}$ and is tangent to $\partial B^4$. The zero set of this vector field on $\overline{B^4}$ consists of a single point, which is a double zero.
\end{example}

\begin{remark}
In general, an ASD PE metric with a nowhere-vanishing Killing field still needs not be globally conformally K\"ahler. Consider a family of hyperbolic metrics on $\mathbb{R}^3\times S^1$
\begin{equation}\label{eq:hyperbolic metric on R3xS1}
h=d\varphi^2+\sinh^2(\varphi)g_{S^2}+\cosh^2(\varphi)\beta^2d\theta^2,
\end{equation}
with Killing field $\mathcal{K}=\partial_\theta$ and a parameter $\beta>0$, where $\theta$ denotes the angle on $\mathbb{S}^1$. One computes
$$\theta_{\mathcal K}=\beta^{2}\cosh^2(\varphi)d\theta,\quad 
d\theta_{\mathcal K}=2\beta^{2}\cosh(\varphi)\sinh(\varphi)\,d\varphi\wedge d\theta,\quad
|(d\theta_{\mathcal K})^{+}|_h
= \sqrt{2}\,\beta\,\sinh(\varphi).
$$
We have $(d\theta_{\mathcal K})^+=0$ iff $\varphi=0$, and $h$ is conformally K\"ahler via $\mathcal{K}$ only away from the circle $\varphi=0$. 
It is worth to note that the fill-in problem around \eqref{eq:hyperbolic metric on R3xS1} was extensively studied by \cite{chang2025poincareeinsteinmanifoldscylindricalconformal,chang2025problemfillingpoincareeinsteinmetric}.
\end{remark}

\subsection{Type II with conformally K\"ahler geometry}
\label{subsec:Type II with Kahler}

Consider a Type II Poincar\'e-Einstein metric $(M,h)$ that is simply-connected. Set $\lambda:=2\sqrt{6}|\mathscr{W}^+_h|_h$ and the conformal change $g:=\lambda^{2/3}h$ is an extremal K\"ahler metric.

From the conformal invariance of self-dual Weyl curvature and the fact that $\mathscr{W}_g^+$ is given by \eqref{eq:self_dual_weyl_curvature_of_a_Kahler_metric}, it is direct to conclude  $|s_g|=\lambda^{1/3}$. In particular, since in the Type II case $\lambda$ never vanishes, we have either $s_g<0$ or $s_g>0$ on $M$. As an extremal K\"ahler metric, $g$ admits a canonical Killing field $\mathcal{K}:=J\nabla_gs_g$, which also turns out to be Killing for the Einstein metric $h$, as the conformal factor $\lambda^{1/3}$ is invariant under $\mathcal{K}$. By employing the K\"ahler reduction for $g$ with respect to $\mathcal{K}$, one can locally write the conformal K\"ahler metric as 
\begin{equation}\label{eq:symplectic reduction}
  g=\zeta^2 h=W(d\zeta^2+e^v(dx^2+dy^2))+W^{-1}\eta^2.
\end{equation}
Here, $\zeta:=s_g$ is the moment map for $\mathcal{K}$, $x+iy$ is a holomorphic coordinate we pick on the K\"ahler reduction. Under \eqref{eq:symplectic reduction}, the Type II Einstein equation reduces to the following equations
\begin{equation}\label{eq:twisted toda}
  (e^v)_{\zeta\zeta}+v_{xx}+v_{yy}=-\zeta We^v,  
\end{equation}
\begin{equation}\label{eq:W}
  W=\frac{12-6\zeta v_\zeta}{12+\zeta^3},
\end{equation}
\begin{equation}\label{eq:d eta}
  d\eta=(We^v)_{\zeta}dxdy+W_xdyd\zeta+W_yd\zeta dx.
\end{equation}
Again, the complex structure $J$ is given by $Jdx=dy$, $Jd\zeta=W^{-1}\eta$. The equation \eqref{eq:twisted toda} is referred to as \emph{twisted $SU(\infty)$ Toda equation}.

\begin{example}
The AdS-Schwarzschild metric is defined as follows, which is Einstein but not ASD:
\begin{equation}\label{eq:AdS-Schwarzchild family}
    h=({1+r^2-\frac{2m}{r}})^{-1}dr^2+({1+r^2-\frac{2m}{r}})\beta^2d\theta^2+r^2g_{S^2}
\end{equation}
where $m>0,\beta>0$ are constants and $\theta$ parametrizes $\mathbb{S}^1$, with $\theta\in[0,2\pi)$. It defines a smooth metric on $\mathbb{R}^2\times S^2$ if and only if $\beta=\frac{2r_m}{1+3r_m^2}$, where $r_m>0$ is the unique real root of $1+r^2-\frac{2m}{r}=0$. Clearly, $r^{-2}h$ is K\"ahler, whose scalar curvature is $\frac{12m}{r}$, thus $h$ is Type II. Take $\zeta=\frac{(12m)^{1/3}}{r}$, one has 
    $$g:=\zeta^2h=Wd\zeta^2+W^{-1}\beta^2d\theta^2+We^{w}g_{S^2}$$
where $W=(1+\zeta^2-2m\zeta^3)^{-1}, w=\log(12m)^{2/3}+\log(1+\zeta^2-2m\zeta^3)$. The K\"ahler metric $g$ provides a conformal compactification.
\end{example}

In the Type II case, up to scaling by constants, $g$ is the only conformal change of $h$ that is K\"ahler. This is because for any other K\"ahler conformal change $g'$, the K\"ahler form $\omega'$ has to be a non-repeated eigen-two-form of $\mathscr{W}_{g'}^+$. It follows that there is a function $f$ such that $\omega'=f\omega$, where $\omega$ is the K\"ahler form for $g$. However, since both $\omega,\omega'$ are closed, $f$ has to be a constant.

\

Now, by our discussion in Section \ref{subsec:Type I with Kahler} and \ref{subsec:Type II with Kahler}, following Definition \ref{def:basic}, we introduce
\begin{definition}\label{def:PE conformally Kahler}
    We say that a Poincar\'e-Einstein manifold $(M,h)$ is conformally K\"ahler, if there exists a function $\zeta\in C^\infty(M)$ such that $g:=\zeta^2h$ is a K\"ahler metric on $M$.  
\end{definition}

Note that for conformally K\"ahler Poincar\'e-Einstein, $\zeta$ is not necessarily a defining function of $\overline{M}$, as illustrated in Example \ref{example:hyperbolic space}. As discussed above, if $h$ is of Type II, then $(M,h)$ is automatically conformally K\"ahler and $\zeta$ is a constant multiple of $|\mathscr{W}_h^+|_h^{1/3}$. If $h$ is of Type I, there might be different choice of $\zeta$, with each choice corresponding to a different Killing field $\mathcal{K}$, and $\zeta = |(d\theta_{\mathcal{K}})^+|_h^{-1}$. To sum up, for a conformally K\"ahler Poincar\'e-Einstein $(M,h)$ with the K\"ahler metric $g=\zeta^2h$, there is an associated Killing field $\mathcal{K}=J\nabla_g\zeta$ and:
\begin{itemize}
    \item if \( h \) is of Type I, then \( \mathcal{K} \) is a Killing field satisfying \( (d\theta_{\mathcal{K}})^+ \neq 0 \) on \( M \);
    \item if \( h \) is of Type II, then $g$ is extremal K\"ahler and up to scaling $\mathcal{K}$ is the extremal vector field, which is Killing.
\end{itemize}

\subsection{Regularity of conformal infinity} 

We now turn to the boundary regularity of conformally K\"ahler Poincar\'e-Einstein manifolds $(M,h)$ and investigate conditions under which $\zeta$ serves as a boundary defining function of $\overline{M}$.

Recall that on a Riemannian manifold, a vector field is called \emph{conformal Killing} if its flow preserves the conformal class of the metric. First, we show the following:

\begin{lemma}\label{lem:conformal Killing fields}
   Let $(M,h)$ be a Poincaré--Einstein 4-manifold and let $\mathcal{K}$ be a Killing field of $h$. Then $\mathcal{K}$ extends smoothly to a conformal Killing field on the geodesic compactification $(\overline{M},g')$. Moreover, $\mathcal{K}$ is tangent to the boundary and induces a conformal Killing field on $(\partial M,g')$.
\end{lemma}

\begin{proof}
Let $g':=dr^{2}+g_r$ be a geodesic compactification. Since $h=r^{-2}g'$, the Killing equation for $\mathcal{K}$ is equivalent to
$$
\mathcal L_{\mathcal K} g' \;=\; 2r^{-1}(\mathcal K r)\, g'. 
$$

Write
$
\mathcal K \;=\; b(r,x)\,\partial_r + T(r,x),
$
with $T$ tangent to the slices $\{r=\text{const}\}$ and $b=\mathcal K r$.  
Take any coordinate system on $\partial M$, and write $T=T^\alpha\partial_\alpha$ one computes (with $II_r:=\tfrac12\partial_r g_r$):
\begin{equation}\label{eq:ckf-b}
\partial_r b = r^{-1}b, 
\end{equation}
\begin{equation}\label{eq:ckf-T}
\partial_r T_\alpha + \nabla^{g_r}_\alpha b - 2\,{II_\alpha}^{\beta}T_\beta=0, 
\end{equation}
\begin{equation}\label{eq:ckf-L}
\mathcal L_T g_r + 2b\,II = 2r^{-1}b\,g_r. 
\end{equation}

Equation \eqref{eq:ckf-b} integrates exactly to
$$
b(r,x)=r\,c(x),
$$
for some smooth function $c(x)$ on $\partial M$. In particular $b(0,x)=0$, so $\mathcal K$ is tangent to $\partial M$.

Because $h$ is Einstein, the first order term in Fefferman--Graham expansion vanishes:
$$
g_r=\hat{g}+r^{2}g_{(2)}+\cdots,
$$
hence $II_r=\tfrac12\partial_r g_r=O(r)$. Inserting $b=r c$ and $II=O(r)$ into \eqref{eq:ckf-T} shows that $T$ extends smoothly to $r=0$. Hence, $\mathcal{K}$ extends smoothly to $\overline{M}$ and is a conformal Killing field for $g'$.
Moreover, restricting \eqref{eq:ckf-L} to boundary $\{r=0\}$ shows that $T(0,\cdot)$ is a conformal Killing field of $(\partial M,g')$.
\end{proof}

Though we will not use the following Lemma, it shows that on any non-hyperbolic Poincar\'e-Einstein manifold, the orbit of the flow of a Killing field must be contained in a compact set.

\begin{lemma}\label{lem:flow to infinity}
  If a Killing field $\mathcal{K}$ flows a point $p$ in $(M,h)$ to infinity, then $(M,h)$ is hyperbolic.
\end{lemma}
\begin{proof}
 Let $ \phi_t $ denote the flow generated by the vector field $ \mathcal{K} $. For any fixed metric ball $ B_N(p) $, we have
$
\sup_{B_N(p)} |\mathscr{W}| = \sup_{\phi_t(B_N(p))} |\mathscr{W}| = \sup_{B_N(\phi_t(p))} |\mathscr{W}|.
$
As $ t \to \infty $, the last quantity tends to zero because $h$ is asymptotically hyperbolic. This implies that $ \mathscr{W} \equiv 0 $.
\end{proof}

\begin{remark}
In the hyperbolic $4$-ball, certain Killing fields flow points to infinity, as illustrated in Example~\ref{example:hyperbolic space}. By contrast, if a Poincar\'e--Einstein $4$-manifold is not hyperbolic, then the closure of flow orbits of a Killing field is contained in a compact set and must form a compact connected abelian Lie group, hence a torus.
\end{remark}

 Next we discuss the Type I case and Type II case separately. We will demonstrate that under a suitable assumption, the conformal K\"ahler metric $g$ provides a natural conformal compactification.

\subsubsection{Type I} 
\begin{proposition}
  In the Type I case, the conformal K\"ahler metric $g$ provides a conformal compactification for the Poincar\'e-Einstein $(M,h)$, if and only if $\mathcal{K}$ has no zeros on $\partial M$.
\end{proposition}

\begin{proof}
Following the proof of Theorem~\ref{prop:Type I with conformal Kahler}, we determine the conformal factor via the formula \( \zeta = |(d\theta_{\mathcal{K}})^+|_h^{-1} \). We begin with a geodesic conformal compactification. By the Fefferman--Graham expansion 
\[
h = \frac{1}{r^2}(dr^2 + g_{(0)} + r^2 g_{(2)} + r^3 g_{(3)} + \ldots),
\]
we obtain
\begin{align*}
\zeta &= \left|\left(d\left(\frac{1}{r^2}\left(\mathcal (\mathcal Kr)dr+g_{(0)}(\mathcal{K}, \cdot) + r^2 g_{(2)}(\mathcal{K}, \cdot) + \ldots \right)\right)\right)^+\right|_h^{-1} \\
&= \left| -\frac{2}{r^3}(dr \wedge \theta_{(0),\mathcal{K}})^+ + \frac{1}{r^2}(d\theta_{(0),\mathcal{K}}+d(\mathcal{K}r)\wedge dr)^+ + (d\theta_{(2),\mathcal{K}})^+ + \ldots \right|_h^{-1} \\
&= r\left| -2(dr \wedge \theta_{(0),\mathcal{K}})^+ + r(d\theta_{(0),\mathcal{K}}+d(\mathcal{K}r)\wedge dr)^+ + r^3(d\theta_{(2),\mathcal{K}})^+ + \ldots \right|_{r^2 h}^{-1} \\
:&=rf
\end{align*}
where \( \theta_{(j),\mathcal{K}} := g_{(j)}(\mathcal{K}, \cdot) \) are smooth 1-forms near \( \partial{M} \). Note that at any point \( p \in \partial M \), 
$
\mathcal{K} \neq 0 \ \Longleftrightarrow\ \theta_{(0),\mathcal{K}} \neq 0 \ \Longleftrightarrow\ (dr \wedge \theta_{(0),\mathcal{K}})^+ \neq 0.
$
If $\mathcal{K}$ has no zeros on $\partial M$, clearly, $f$ is a smooth positive function on $\overline{M}$, thus \( \zeta \) is a smooth defining function.

Conversely, if \( \mathcal{K}(p) = 0 \) for some \( p \in \partial M \), then $f^{-1}=O(r)$ along the normal direction from $p$, so $\zeta$ is not a defining function.  
\end{proof}

Note that it is possible to have a conformally K\"ahler Type I Poincar\'e-Einstein, where the conformal K\"ahler metric \( g \) does not yield a conformal compactification as illustrated in Example \ref{example:hyperbolic space}.

\subsubsection{Type II}

We have

\begin{proposition}
  In the Type II case, the conformal K\"ahler metric $g$ always provides a conformal compactification for the Poincar\'e-Einstein $(M,h)$.
\end{proposition}
\begin{proof}
  Pick a geodesic compactification $h=\frac{1}{r^2}(dr^2+g_r)$ and denote $g'=dr^2+g_r$. By some computations, there is an expansion for $\mathscr{W}^+_{g'}$ under $g'$ (see Proposition 6.4 in \cite{gurskyselfdual}, for example): $\mathscr{W}^+_{g'}=r\cdot W_1^++r^2\cdot W_2^++O(r^3)$, where $W_j^+$ are smooth tensors defined on $\partial M$. Hence, $U:=|\mathscr{W}_{g'}^+|_{g'}^2/r^2$ is a smooth function on $\overline{M}$, positive on $M$. It follows that
  $$\lambda^{1/3}=(2\sqrt{6}|\mathscr{W}^+_h|_h)^{1/3}=(2\sqrt{6}|\mathscr{W}^+_{g'}|_{g'}\cdot r^2)^{1/3}=(2\sqrt{6})^{1/3}U^{1/6}r\leq Cr.$$

  On the other hand, since $\mathcal{K}$ is smooth on $\overline{M}$ and $J\mathcal{K}=\pm\nabla _h \lambda^{-1/3}=\pm r^2\nabla_{g'}\lambda^{-1/3}$, it follows that $|\nabla_{g'}\lambda^{-1/3}|_{g'}\leq C/r^2$. Thus, near $\partial M$, we have
  $$\lambda^{-1/3}\leq C/r.$$
  
  Hence, $C^{-1}r\leq \lambda^{1/3}\leq Cr$ near $\partial M$, so $U$ is a positive on $\overline{M}$ and $\lambda^{1/3}$ is a smooth defining function. Thus $(M,g)$ is a conformal compactification.
\end{proof}

\begin{remark}
    The proof implies that $\mathcal{K}$ is nowhere-vanishing on $\partial M$.
\end{remark}

\

Now summarizing the above, we can conclude this subsection as follows. For a Poincar\'e-Einstein $(M,h)$ and a conformal change $g=\zeta^2h$ that is K\"ahler, there is the associated Killing field $\mathcal{K}=J\nabla_g\zeta$.
\begin{definition}\label{def:regular conformally Kahler}
  A conformally K\"ahler Poincar\'e-Einstein $(M,h)$ with its associated K\"ahler metric $g=\zeta^2h$ is called \emph{regular}, if the associated Killing field $\mathcal{K}$ is tangent to $\partial M$ and induces a free $\bS^1$-action on $\partial M$.
\end{definition}
Note that a priori there is no canonical choice of $g$ since there still is the ambiguity of scaling by constants, but later in Section \ref{subsec:canonical conformal change} we will specify a canonical one.
\begin{proposition}
  For a regular conformally K\"ahler Poincar\'e-Einstein $(M,h)$ and an associated conformal K\"ahler metric $g$, the K\"ahler metric $g$ provides a natural smooth conformal compactification. 
\end{proposition}

In the Type I case the conformal compactification $(\overline{M},g)$ is K\"ahler scalar-flat. In the Type II case it is strictly extremal K\"ahler. There is indeed Poincar\'e-Einstein that has irregular conformally K\"ahler geometry, where $\mathcal{K}$ either induces an $\bS^1$-action with non-trivial discrete isotropy ($\mathcal{K}=a(x_1\partial_{x_2}-x_2\partial_{x_1})+b(x_3\partial_{x_4}-x_4\partial_{x_3})$ with co-prime integers $a,b$ in Example \ref{example:hyperbolic space}), or only induces an $\mathbb{R}$-action ($\mathcal{K}=a(x_1\partial_{x_2}-x_2\partial_{x_1})+b(x_3\partial_{x_4}-x_4\partial_{x_3})$ with $a/b$ irrational in Example \ref{example:hyperbolic space}). More complicated examples are AdS Kerr metrics.

We restrict ourselves to the regular case for the rest of this paper unless explicitly mentioned.

\subsection{The canonical conformal change}
\label{subsec:canonical conformal change}

We introduce a canonical conformal change for regular conformally K\"ahler Poincar\'e-Einstein manifolds $(M,h)$. Note that the K\"ahler metric $g=\zeta^2h$ still has the ambiguity of scaling by constants.  Over the end of $(M,g)$, since $\mathcal{K}=J\nabla_g\zeta$ has no zero and induces a free $\bS^1$-action, we can apply K\"ahler reduction to write the metric $g=W(d\zeta^2+e^v(dx^2+dy^2))+W^{-1}\eta^2$ with $v,W,\eta$ satisfying \eqref{eq:Type I Toda}-\eqref{eq:Type I d eta} or \eqref{eq:twisted toda}-\eqref{eq:d eta}. Denote the Riemann surface obtained from the K\"ahler reduction as $\Sigma$. Pick a metric $g_{\Sigma}$ over $\Sigma$ with constant curvature $K_{\Sigma}=1,0,-1$, which is normalized to have volume $4\pi^2$ if $K_\Sigma=0$. Since $W$ is globally defined and $We^v(dx^2+dy^2)$ is the metric arising from K\"ahler reduction on $\Sigma$, $e^v(dx^2+dy^2)$ is a well-defined metric on $\Sigma$ as well. Write $e^v(dx^2+dy^2)=e^wg_{\Sigma}$, where $w$ is a function over $\Sigma$, which depends on $\zeta$.

We first discuss the Type II case. Let $k\in\mathbb{R}\setminus\{0\}$ be the constant such that 
$$k\cdot (2\sqrt{6}\max|\mathscr{W}_h^+|_h)^{1/3}=\frac{1}{2}\sign(s_g).$$
Recall that in Section \ref{subsec:Type II with Kahler} the Ans\"atze \eqref{eq:symplectic reduction}-\eqref{eq:d eta} are formulated with $\zeta=\lambda^{1/3}$ and $g=\lambda^{2/3}h$.
Now set $g_\sharp := k^2 g = k^2 \zeta^2 h$. Notice that $k\in\mathbb{R}\backslash\{0\}$ is chosen so that 
\begin{equation}\label{eq:definition of k Type II}
    k \zeta_{\max} = \tfrac{1}{2} \quad \text{if } s_g>0 \text{ on } M, 
\qquad\text{or}\qquad 
k \zeta_{\min} = \tfrac{1}{2} \quad \text{if } s_g<0 \text{ on } M.
\end{equation}
Write $\xi:=k\zeta,v_\sharp:=v+\log k^2,\eta_\sharp=k\eta$. The Ans\"atze \eqref{eq:symplectic reduction}-\eqref{eq:d eta} over the end of $(M,h)$ now become
\begin{equation}\label{eq:g first reduction}
  g_\sharp=W(d\xi^2+e^{v_\sharp}(dx^2+dy^2))+W^{-1}\eta_\sharp^2,
\end{equation}
\begin{equation}\label{eq:twisted toda first reduction}
  e^{v_\sharp}_{\xi\xi}+(v_{\sharp})_{xx}+(v_{\sharp})_{yy}=-\xi e^{v_\sharp}\frac{12-6\xi \partial_{\xi}v_{\sharp}}{12k^3+\xi^3},
\end{equation}
\begin{equation}\label{eq:W first reduction}
  W=\frac{12-6\xi \partial_{\xi}v_{\sharp}}{12+\xi^3/k^3},
\end{equation}
\begin{equation}\label{eq:d eta first reduction}
  d\eta_\sharp=(We^{v_\sharp})_{\xi}dxdy+W_xdyd\xi+W_yd\xi dx.
\end{equation}
We can further simplify  by writing $e^{v_\sharp}(dx^2+dy^2)$ as $e^{w_\sharp}g_{\Sigma}$. Then the Ans\"atze \eqref{eq:g first reduction}-\eqref{eq:d eta first reduction} become
\begin{equation}\label{eq:g final}\tag{ansatz-$1_k$}
  g_\sharp=W(d\xi^2+e^{w_\sharp}g_{\Sigma})+W^{-1}\eta_\sharp^2,
\end{equation}
\begin{equation}\label{eq:twisted toda final}\tag{ansatz-$2_k$}
  e^{w_\sharp}_{\xi\xi}+\Delta_{\Sigma}{w_\sharp}-2K_{\Sigma}=-\xi e^{w_\sharp}\frac{12-6\xi \partial_{\xi}{w_\sharp}}{12k^3+\xi^3},
\end{equation}
\begin{equation}\label{eq:W final}\tag{ansatz-$3_k$}
  W=\frac{12-6\xi \partial_{\xi}{w_\sharp}}{12+\xi^3/k^3},
\end{equation}
\begin{equation}\label{eq:d eta final}\tag{ansatz-$4_k$}
  d\eta_\sharp=\star ((d+\partial_\xi{w_\sharp} d\xi)W).
\end{equation}
Here, in \eqref{eq:d eta final}, $\star$ refers to the Hodge star of $d\xi^2+e^{w_\sharp}g_{\Sigma}$. The function $\xi$ is the moment map for the Hamiltonian Killing field $\mathcal{K}_\sharp:=\frac{1}{k}\mathcal{K}$. The K\"ahler form of $g_\sharp$ is $\omega_{\sharp}:=d\xi \eta_\sharp+e^{w_\sharp}\omega_{\Sigma}$. The scalar curvature of $g_\sharp$ is $s_{g_\sharp}=\frac{1}{k^3}\xi$ in the Type II case.

As for the Type I case, starting with any K\"ahler $g=\zeta^2h$, we again rescale $\zeta$ to $\xi$  so that
\begin{equation}\label{eq:definition of k Type I}
     \xi_{\max}=\tfrac12.
\end{equation}
This normalization is chosen to align with the Type I scaling: setting $k=\infty$ (or equivalently, $k=-\infty$) in \eqref{eq:g final}--\eqref{eq:d eta final} reproduces exactly the ansatz \eqref{eq:Type I metric}--\eqref{eq:Type I d eta}.

In the following, we shall exclusively work with canonical infinity and the above Ans\"atze, and ignore the lower index $\sharp$ for simplicity. 

\begin{definition}\label{def:canonical infinity}
  For a Poincar\'e-Einstein $(M,h)$ with regular conformally K\"ahler geometry, its \emph{canonical conformal change} is defined as the K\"ahler metric $g:=g_\sharp$ above, and its \emph{canonical infinity} is defined as $g^{\flat}:=g|_{\partial M}$. There is also the \emph{2-dimensional infinity} $g^{\natural}$, by which we mean the metric on $\Sigma$ obtained by taking quotient of $g^\flat$ by the $\bS^1$-action. We shall denote regular conformally K\"ahler Poincar\'e-Einstein with its canonical conformal change as a triple $(M,h,g)$.
\end{definition}

Note that from \eqref{eq:W final}, $W\equiv1$ on the boundary $\partial M$, so the Killing field $\mathcal{K}$ over $\partial M$ has constant size one. The canonical infinity $g^{\flat}=e^wg_{\Sigma}+\eta^2$. Under our convention here, the standard hyperbolic metric $(B^4,h)$ with conformally K\"ahler geometry given by the Killing field $\mathcal{K}$ rotating the Hopf fibration has canonical conformal change
$$g=W(d\xi^2+e^wg_{S^2})+W^{-1}\eta^2$$
with $W=\frac{1}{1-2\xi},e^w=\frac{1}{4}(1-2\xi)^2$, and $\xi\in[0,\frac{1}{2}]$. The canonical infinity is $g^\flat=\frac{1}{4}g_{S^2}+\eta^2=g_{S^3}$.

\begin{remark}\label{remark:gauge}
  In the above formulation \eqref{eq:g final}-\eqref{eq:d eta final}, there is the issue of choice of gauge when $\Sigma=S^2$. For any automorphism $\Phi:S^2\to S^2$ that preserves the complex structure, the metric $\Phi^*g_{S^2}$, which can be written as $e^\psi g_{S^2}$ for a function $\psi$ as $\Phi^*g_{S^2}$ is conformal to $g_{S^2}$, still has constant curvature. There is no such issue when $\ttg\geq1$. 
\end{remark}

\begin{remark}\label{remark:AdS-Schwarzchild}
In the AdS--Schwarzschild family, for each $\beta \in \bigl(0,\tfrac{1}{\sqrt{3}}\bigr)$ there exist two distinct values of $m$ in \eqref{eq:AdS-Schwarzchild family}, both yielding Poincar\'e--Einstein metrics whose conformal infinity is the conformal class of the product metric $\mathbb{S}^2 \times \mathbb{S}^1_\beta$, where $\mathbb{S}^2$ is the round $S^2$. However, these two Type~II Poincar\'e--Einstein metrics have different canonical infinity, which in fact differ by a rescaling.
\end{remark}

\subsection{Morse-Bott property for \texorpdfstring{$\xi$}{}}
\label{subsec:Morse-Bott}

We have the observation by LeBrun that $\xi$ is a Morse-Bott function \cite{lebrun20}.
\begin{lemma}
  The function $\xi$ over $M$ is a Morse-Bott function. More precisely, the critical locus $\nabla_g\xi=0$ is a disjoint union $\sqcup C_j$ of compact submanifolds, and the Hessian $\mathrm{Hess}_g\xi$ is non-degenerate on the normal bundle $(TC_j)^{\perp}$ for each $C_j$. Each $C_j$ is either a single point, or a real 2-dimensional surface.
\end{lemma}
\begin{proof}
  This is exactly Lemma 1 in \cite{lebrun20} since $\mathcal{K}=J\nabla_g\xi$ is a Hamiltonian Killing field.
\end{proof}

There can only be isolated critical points or critical surface for the Morse-Bott function $\xi$.
Computing the conformal change of the scalar curvature we obtain
\begin{equation}
  -12=\xi^2s_g+6\xi\Delta_g\xi-12|\nabla_g\xi|_g^2.
\end{equation}
Therefore, we have:
\begin{itemize}
  \item Each isolated critical point must be either a local maximum, or a saddle point.
  \item Each critical surface $C$ must be a local maximum.
\end{itemize}
In the saddle point case, there is an orthogonal decomposition of the tangent space into the direct sum of two $J$-invariant 2-dimensional subspaces, where $\mathrm{Hess}_g\xi$ is positive definite on the first subspace and negative definite on the second subspace. Applying the Morse theory on $\xi$, starting from the maximum $\xi=\frac{1}{2}$ to the minimum $\xi=0$, one further obtains that up to homotopy $M$ is gotten by adding
\begin{itemize}
  \item a disjoint, unattached point for each isolated local maximum;
  \item a disjoint, unattached Riemann surface for each non-isolated local maximum;
  \item a 2-disk, attached along its boundary, for each saddle point.
\end{itemize}
See the proof of Proposition 3 of \cite{lebrun20}. Because we are only attaching path-connected spaces along with their path-connected boundaries, there only exists one global maximum for $\xi$ that is either a critical point or a critical surface, and all other critical points are isolated saddle points. There is no other critical surface and isolated local maximum.

If there is no saddle point, beginning with the maximum, which is either a point or a surface, up to homotopy, the manifold $M$ is obtained by purely attaching two cells.
As a particular consequence, when $M$ is $B^4$,
\begin{corollary}\label{cor:topology over point}
    For a Poincar\'e-Einstein $(B^4,h)$ with regular conformally K\"ahler geometry, the Killing field $\mathcal{K}$ has exactly one isolated fixed point.
\end{corollary}
When $M$ is diffeomorphic to a line bundle over a Riemann surface $\Sigma_{\ttg}$ with genus $\ttg$, we also have 
\begin{corollary}\label{cor:topology over riemann surface}
    For a Poincar\'e-Einstein $(M,h)$ with regular conformally K\"ahler geometry, where the underlying manifold $M$ is diffeomorphic to a complex line bundle over $\Sigma_{\ttg}$, the Killing field $\mathcal{K}$ only has a fixed surface with genus $\ttg$. 
\end{corollary}
The fixed surface, denoted by $\Xi$, is the \emph{bolt} of the $\bS^1$-action. At the level of complex structure, there is a complex line bundle $L\to\Sigma_{\ttg}$, and in this situation $M$ is a subdomain of $L$ containing the bolt $\Xi\simeq\Sigma_{\ttg}$, invariant under the natural rotation real holomorphic vector field, which up to scaling is $\mathcal{K}$.

\section{Decoupled solutions and cohomogeneity-one Einstein metrics}
\label{sec:decoupled solutions}

A class of solutions to \eqref{eq:twisted toda final} are the \emph{decoupled solutions}. We begin with some simple observation. Consider a regular conformally K\"ahler PE 4-manifold $(M,h,g)$, where the critical set of the Morse-Bott function $\xi$ consists of only one maximum point or only one maximum surface. The underlying manifold $M$ is then either $B^4$, or a complex line bundle over a surface. At non-critical values, level set of $\xi$ is an $\mathbb{S}^1$-bundle over the corresponding K\"ahler reduction. All the K\"ahler reductions are biholomorphic, denoted as a Riemann surface $\Sigma$.  A proper constant multiple of $d\eta$ is the curvature of the $\mathbb{S}^1$-bundle. Suppose the period of $\mathcal{K}$ is $\mathfrak{p}>0$, by which we mean the time $\mathfrak{p}$ flow of $\mathcal{K}$ exactly gives the primitive $\bS^1$-action. Since $W\equiv1$ over $\partial M$, the period $\mathfrak{p}$ also has the geometric meaning of the length of the $\mathbb{S}^1$-fibers over $\partial M$. Then we have
$$\frac{1}{\mathfrak{p}}\int_{\Sigma}d\eta=\text{degree of the $\mathbb{S}^1$-bundle}.$$
Let ${deg}$ denotes the degree of the bundle. If $M$ is the total space of a complex line bundle, then $deg$ coincides with the degree of the bundle; whereas if $M$ is $B^4$, we have $deg=-1$.
From \eqref{eq:d eta} it follows
$$\int_{\Sigma}(We^w)_\xi d\mathrm{vol}_{\Sigma}={deg}\cdot\mathfrak{p}.$$
Consequently
\begin{equation}\label{eq:integral We^w}
  \int_{\Sigma}We^wd\mathrm{vol}_{\Sigma}=({deg}\cdot\mathfrak{p})\xi+\mathfrak{a}
\end{equation}
for a constant $\mathfrak{a}$. Evaluating the above at $\xi=0$ we see that $\mathfrak{a}$ is exactly the area of the K\"ahler reduction over the boundary $\partial M$. Next noticing that \eqref{eq:twisted toda final} is exactly $(e^{w})_{\xi\xi}+\Delta_{\Sigma} w-2K_\Sigma=-\xi W e^w\frac{1}{k^3}$, one can integrate it over $\Sigma$ to obtain
$$\int_{\Sigma} (e^w)_{\xi\xi}d\mathrm{vol}_\Sigma-4\pi\chi(\Sigma)=-\frac{1}{k^3}(({deg}\cdot\mathfrak{p})\xi+\mathfrak{a})\xi.$$
As a consequence,
$$\int_{\Sigma}e^wd\mathrm{vol}_{\Sigma}=-\frac{{deg}\cdot\mathfrak{p}}{12k^3}\xi^4-\frac{\mathfrak{a}}{6k^3}\xi^3+2\pi\chi(\Sigma)\xi^2+A\xi+B.$$
Using \eqref{eq:W final} and \eqref{eq:integral We^w} we can compute the constants $A$ and $B$ explicitly and get
\begin{equation}\label{eq:integral e^w}
  \int_{\Sigma}e^wd\mathrm{vol}_{\Sigma}=-\frac{{deg}\cdot\mathfrak{p}}{12k^3}\xi^4-\frac{\mathfrak{a}}{6k^3}\xi^3+2\pi\chi(\Sigma)\xi^2+2({deg}\cdot\mathfrak{p})\xi+\mathfrak{a}.
\end{equation}

With \eqref{eq:integral We^w} and \eqref{eq:integral e^w}, it is direct to write down all the solutions to \eqref{eq:twisted toda final} that takes the form $w=w_1+w_2$, namely the \emph{decoupled solutions}, where $w_1$ is a function of $\xi$ and $w_2$ is a function on the base $\Sigma$. Computation gives
$$(e^{w_1})_{\xi\xi}+\xi e^{w_1}\frac{12-6\xi\partial_{\xi}w_1}{12k^3+\xi^3}=-e^{-w_2}(\Delta_{\Sigma}w_2-2K_\Sigma)\equiv c$$
with $c$ being a constant, having the same sign as $K_\Sigma$ (if $K_\Sigma=0$ then also $c=0$). This is equivalent to $(e^{w_1})_{\xi\xi}+\xi \frac{12e^{w_1}-6\xi\partial_{\xi}(e^{w_1})}{12k^3+\xi^3}=c$ and $e^{w_2}g_{\Sigma}$ has constant curvature $\frac{1}{2}c$. It follows that the decoupled solutions take the form
\begin{equation}
  e^{w}=\frac{1}{\mathrm{Vol}_{\Sigma}}\left(-\frac{{deg}\cdot\mathfrak{p}}{12k^3}\xi^4-\frac{\mathfrak{a}}{6k^3}\xi^3+2\pi\chi(\Sigma)\xi^2+2({deg}\cdot\mathfrak{p})\xi+\mathfrak{a}\right)e^{\psi}
\end{equation}
where $\psi$ is any function on the base $\Sigma$ such that $e^\psi g_{\Sigma}$ has the same curvature $K_\Sigma$ and volume $\mathrm{Vol}_{\Sigma}$; in particular, $\psi$ is the logarithmic determinant of the Jacobian of a conformal map of $\Sigma$ and such non-trivial $\psi$ exists only when $\Sigma=S^2$. By changing our choice of the $g_{\Sigma}$, one can further simplify to
\begin{equation}\label{eq:decoupled solutions w}
  e^{w}=\frac{1}{\mathrm{Vol}_{\Sigma}}\left(-\frac{{deg}\cdot\mathfrak{p}}{12k^3}\xi^4-\frac{\mathfrak{a}}{6k^3}\xi^3+2\pi\chi(\Sigma)\xi^2+2({deg}\cdot\mathfrak{p})\xi+\mathfrak{a}\right)
\end{equation}
with
\begin{equation}\label{eq:decoupled solutions W}
  W=\frac{({deg}\cdot\mathfrak{p})\xi+\mathfrak{a}}{-\frac{{deg}\cdot\mathfrak{p}}{12k^3}\xi^4-\frac{\mathfrak{a}}{6k^3}\xi^3+2\pi\chi(\Sigma)\xi^2+2({deg}\cdot\mathfrak{p})\xi+\mathfrak{a}}.
\end{equation}

Locally, these decoupled solutions define exactly the same Einstein metrics of Page-Pope \cite{page} in dimension 4. By considering the smoothness condition inside, one can classify the Poincar\'e-Einstein manifolds $(M,h)$ with regular conformally K\"ahler geometry, which are decoupled in the above sense. 
These particularly include classical examples like the AdS-Schwarzschild (conformal infinity $\mathbb{S}^1_\beta\times \mathbb{S}^2$), Pedersen \cite{Pedersen1986} (conformal infinity Berger spheres).

In Sections \ref{subsec:T2 decouple}–\ref{subsec:S2 decouple}, we establish necessary conditions for the decoupled solutions to correspond to smooth Einstein metrics. In Section \ref{subsec:conclusion}, we prove that these conditions are also sufficient, and we conclude by summarizing all decoupled solutions.

\subsection{Decoupled solutions with base \texorpdfstring{$T^2$}{}}
\label{subsec:T2 decouple}

Consider the case that $\Sigma$ is a Riemann surface $T^2$, hence the manifold $M$ is the total space of a complex line bundle over the base $T^2$. Recall that in this case we equip $\Sigma$ with the flat metric $g_\Sigma$ of volume $\mathrm{Vol}_\Sigma=4\pi^2$. First, we must have $We^w=\frac{1}{4\pi^2}(({deg}\cdot\mathfrak{p})\xi+\mathfrak{a})>0$ over $[0,\frac{1}{2}]$ since $We^wg_{\Sigma}$ is the K\"ahler reduction metric on the base $T^2$. Moreover, at $\xi=\frac{1}{2}$, the Killing field $\mathcal{K}$ vanishes, so we have $W^{-1}(\frac12)=0$, which now gives
\begin{equation}\label{eq:T2 smooth condition 1}
  -\frac{{deg}\cdot\mathfrak{p}}{12k^3}\frac{1}{16}-\frac{\mathfrak{a}}{6k^3}\frac{1}{8}+2({deg}\cdot\mathfrak{p})\frac{1}{2}+\mathfrak{a}=0.
\end{equation}
In particular, this implies $k\neq-\frac{1}{\sqrt[3]{96}}$, as $\frac{1}{2}deg\cdot\mathfrak{p}+\mathfrak{a}>0$.
Further, since $g=Wd\xi^2+W^{-1}\eta^2+We^wg_{\Sigma}$ and the metric $Wd\xi^2+W^{-1}\eta^2$ on each $\mathbb{C}^*$ closes up to be $\mathbb{C}$ at $\xi=\frac12$, from \eqref{eq:decoupled solutions W} we have $W=A_0(1/2-\xi)^{-1}+O(1)$ near $\xi=\frac12$ with a constant $A_0>0$. Applying \eqref{eq:decoupled solutions w} and \eqref{eq:decoupled solutions W} we can compute $A_0$ explicitly as follows
\begin{equation}\label{eq:T2 smooth condition 2}
	A_0=-\frac{\frac{1}{2}{deg}\cdot \mathfrak{p}+\mathfrak{a}}{-\frac{{deg}\cdot\mathfrak{p}}{24k^3}-\frac{\mathfrak{a}}{8k^3}+2{deg}\cdot\mathfrak{p}}.
\end{equation}
Implicitly, in the above we require $-\frac{{deg}\cdot\mathfrak{p}}{24k^3}-\frac{\mathfrak{a}}{8k^3}+2{deg}\cdot\mathfrak{p}\neq0$, which is equivalent to $\xi=\frac{1}{2}$ being a non-repeated zero of $-\frac{{deg}\cdot\mathfrak{p}}{12k^3}\xi^4-\frac{\mathfrak{a}}{6k^3}\xi^3+2({deg}\cdot\mathfrak{p})\xi+\mathfrak{a}=0$.
To ensure $Wd\xi^2+W^{-1}\eta^2$ on each $\mathbb{C}^*$  closing up to a smooth metric on $\mathbb{C}$ without cone angle, we also have the following relation between $A_0$ and the period $\mathfrak{p}$ of the Killing field
\begin{equation}\label{eq:T2 smooth condition 3}
  \frac{\mathfrak{p}}{2A_0}=2\pi.
\end{equation}
Summarizing \eqref{eq:T2 smooth condition 1}-\eqref{eq:T2 smooth condition 3} one obtains that 
\begin{itemize}
  \item If ${deg}\neq0$, then
\begin{equation}\label{eq:T2 decoupled period}
      \mathfrak{p} = \frac{96 \pi k^3}{96 k^3 + 1}, \quad 
    \mathfrak{a} = deg \cdot \frac{24 \pi k^3 (1 - 192 k^3)}{(48 k^3 - 1)(96 k^3 + 1)}. 
\end{equation}
Further computation gives
  \begin{align}
    e^w&=\frac{{deg}}{4\pi^2}\cdot\frac{4 \pi  (2\xi-1) \left(1152 k^6-6 k^3 \left(8 \xi^3-12 \xi^2-6
    \xi+1\right)+\xi^3\right)}{\left(48 k^3-1\right) \left(96 k^3+1\right)},\label{eq:T2 decoupled e^w}\\
    We^w&=\frac{{deg}}{4\pi^2}\cdot \frac{24 \pi  k^3 \left(192 k^3 (\xi-1)-4 \xi+1\right)}{\left(48 k^3-1\right) \left(96
    k^3+1\right)}.\label{eq:T2 decoupled W}
  \end{align}
  \item If ${deg}=0$, then $k=\frac{1}{\sqrt[3]{48}}$, $\mathfrak{p}=\frac{2}{3}\pi$, and $\mathfrak{a}$ can take arbitrary positive value. Further computation gives
  \begin{align}
    e^w&=\frac{\mathfrak{a}}{4\pi^2}(1-8\xi^3),\label{eq:decoupled e^w T^2 degree 0}\\
    We^w&=\frac{\mathfrak{a}}{4\pi^2}.\label{eq:decoupled We^w T^2 degree 0}
  \end{align}
\end{itemize}

Now we will precisely determine the range of parameters. In order for the decoupled solution to define a smooth metric, it is necessary to impose the following conditions, which we shall later verify to be sufficient as well.
\begin{enumerate}[label=\textbf{(C\arabic*)},ref=\textbf{(C\arabic*)}]
  \item \label{C1} $\mpp>0$ and $\ma>0$.
  \item \label{C2} $We^w>0$ for all $\xi\in\bigl[0,\tfrac12\bigr]$.
  \item \label{C3} $e^w>0$ on $\bigl[0,\tfrac12\bigr)$ and $e^w$ has a simple zero at $\xi=\tfrac12$.
\end{enumerate}

We shall show that conditions \ref{C1}--\ref{C3} determine precisely one of the ranges listed below. 
Since $e^w$ is, in general, a quartic polynomial in $\xi$, the key step in the following proofs is to verify that the range determined by \ref{C1}--\ref{C2} guarantees $e^w>0$ on $\bigl[0,\tfrac{1}{2}\bigr)$.

\begin{proposition}[$T^2$ base]\label{prop:range for k T2}
    The decoupled solution defines a smooth metric only if either
    \begin{itemize}
        \item $k\in[-\infty, -\frac{1}{\sqrt[3]{96}})\cup (\frac{1}{\sqrt[3]{48}},\infty]$ and $deg<0$;
        \item $k=\frac{1}{\sqrt[3]{48}}$, $deg=0$, $\mathfrak{a}>0$ is arbitrary;
        \item $k\in(\frac{1}{\sqrt[3]{192}},\frac{1}{\sqrt[3]{48}})$ and $deg>0$.
    \end{itemize}
\end{proposition}
\begin{proof}
When $deg = 0$, conditions \ref{C1}--\ref{C3} are clearly satisfied by 
\eqref{eq:decoupled e^w T^2 degree 0}--\eqref{eq:decoupled We^w T^2 degree 0}. 
From now on, we assume $deg \neq 0$. The cases $k = \pm\infty$ are immediate and thus omitted, 
so we focus on the case $k \neq \pm\infty$. Since $We^w$ is linear in $\xi$, its minimum is achieved at the endpoints, so condition \ref{C2} is equivalent to $ \mathfrak{a}>0 $ and $ \tfrac{1}{2}deg\cdot \mathfrak{p}+\mathfrak{a}>0 $, where the latter inequality is 
$
-\frac{\pi}{2}\frac{deg}{1-\frac{1}{48k^3}}>0.
$
Hence conditions \ref{C1}--\ref{C2} are equivalent to
\begin{equation}
\frac{96 \pi k^3}{96 k^3+1}>0, \quad deg\cdot\frac{24 \pi k^3 (1-192 k^3)}{(48 k^3-1)(96 k^3+1)}>0, \quad -\frac{\pi}{2}\frac{deg}{1-\frac{1}{48k^3}}>0.
\end{equation}
which are equivalent to $deg<0,k\in(-\infty,-\frac{1}{\sqrt[3]{96}})\cup(\frac{1}{\sqrt[3]{48}},\infty)$, or $deg>0, k\in(\frac{1}{\sqrt[3]{192}},\frac{1}{\sqrt[3]{48}})$.
We now verify these imply \ref{C3}. 

\medskip

\textbf{Case 1}: $deg<0$ and $k\in(-\infty,-\frac{1}{\sqrt[3]{96}})\cup(\frac{1}{\sqrt[3]{48}},\infty)$.

Let $A(\xi)=1152k^6-6k^3(8\xi^3-12\xi^2-6\xi+1)+\xi^3$ with $A'(\xi)=3\xi^2+144k^3\xi+(36k^3-144k^3\xi^2)$. If $k>\frac{1}{\sqrt[3]{48}}$, then $A'(\xi)=36k^3+144k^3\xi(1-\xi)+3\xi^2>0$, so $\min_{\xi\in[0,\frac12]} A=A(0)=6k^3(192k^3-1)>0$. If $k<-\frac{1}{\sqrt[3]{96}}$, then $A'(\xi)$ is convex and negative at the endpoints, hence $A'(\xi)<0$ and $\min_{\xi\in[0,\frac12]}  A=A(1/2)=1152(k^3+\frac{1}{96})^2>0$. In either case $e^w>0$ on $[0,1/2)$. 

\medskip

\textbf{Case 2}: $deg>0$ and $k\in(\frac{1}{\sqrt[3]{192}},\frac{1}{\sqrt[3]{48}})$.

We have $A'(\xi)>0$ and $\min_{\xi\in[0,\frac12]}  A=A(0)=6k^3(192k^3-1)>0$, so again $e^w>0$ on $[0,1/2)$.
\end{proof}

\subsection{Decoupled solutions with base \texorpdfstring{$\Sigma_{\ttg}$}{} when \texorpdfstring{$\ttg> 1$}{}}
\label{subsec:hyperbolic decouple}

Next consider the case $\Sigma=\Sigma_{\ttg}$ where the underlying manifold $M$ is the total space of a complex line bundle over $\Sigma_{\ttg}$. Denote the Euler characteristic $\chi=2-2\ttg<0$. The volume of $(\Sigma_{\ttg},g_{\Sigma_{\ttg}})$ is $\mathrm{Vol}_{\ttg}=-2\pi\chi=4\pi(\ttg-1)$ when $\ttg> 1$.

Again, for decoupled solutions we have 
$$We^w=\frac{1}{\operatorname{Vol}_{{\ttg}}}((deg\cdot \mpp)\xi+\ma)>0$$
over $[0,\frac{1}{2}]$. At $\xi=\frac{1}{2}$, we  have $W^{-1}(\frac{1}{2})=0$, which gives
\begin{equation}\label{eq:e^w vanishes at 1/2}
    -\frac{{deg}\cdot\mathfrak{p}}{12k^3}\frac{1}{16}-\frac{\mathfrak{a}}{6k^3}\frac{1}{8}+2\pi\chi \frac{1}{4}+{deg}\cdot\mathfrak{p}+\mathfrak{a}=0.
\end{equation}
Similar to the $T^2$ case, to ensure the metric $Wd\xi^2+W^{-1}\eta^2$ on each $\mathbb{C}^*$ closing up to $\mathbb{C}$ at $\xi=\frac{1}{2}$, we have $W\sim A_0(\frac{1}{2}-\xi)^{-1}$ near $\xi=\frac{1}{2}$ and we can compute $A_0$ as 
$$A_0=-\frac{\frac{1}{2}{deg}\cdot \mathfrak{p}+\mathfrak{a}}{-\frac{{deg}\cdot\mathfrak{p}}{24k^3}-\frac{\mathfrak{a}}{8k^3}+2\pi\chi+2{deg}\cdot\mathfrak{p}}.$$
To make sure there is no cone angle, we again need 
\begin{equation}\label{eq:mpp relation for g>1}
    \frac{\mpp}{2A_0}=2\pi
\end{equation} 
Therefore
\begin{itemize}
    \item If $k\neq\frac{1}{\sqrt[3]{48}}$, then
    \begin{equation}\label{eq:ma relation for g>1}
        \ma=-\frac{deg\cdot\mpp(1-\frac{1}{192k^3})+\frac{\pi}{2}\chi}{1-\frac{1}{48k^3}}.
    \end{equation}
    By \eqref{eq:mpp relation for g>1}-\eqref{eq:ma relation for g>1}, we have 
    $$\left((\frac{1}{48k^3}+2)\mpp-2\pi\right)\left((\frac{1}{48k^3}+2){deg}\cdot \mpp+2\pi\chi\right)=0.$$
    Hence, we must have $k\neq-\frac{1}{\sqrt[3]{96}}$, and we either have
    \begin{equation}\label{eq:decoupled p and a for g>1}
        \mpp=\frac{96\pi k^3}{96k^3+1},\quad \ma=-\frac{\pi(\chi+deg)}{2(1-\frac{1}{48k^3})}-\frac{\pi deg}{2+\frac{1}{48k^3}},
    \end{equation}
    or
    $$deg\cdot\mpp =-\frac{96\pi \chi k^3}{96k^3+1},\quad \ma=\frac{48\pi \chi k^3}{1+96k^3}.$$
    However, at $\xi=\frac{1}{2}$, we have $\frac{1}{2}deg\cdot\mpp+\ma>0$, which implies the second case cannot happen. Therefore, if $k\neq \frac{1}{\sqrt[3]{48}}$ we must have \eqref{eq:decoupled p and a for g>1}. It also follows that $deg\neq-\chi$, as otherwise we would conclude $\frac{1}{2}deg\cdot\mpp+\ma=0$ from \eqref{eq:decoupled p and a for g>1}, which is impossible. 
    \item If $k=\frac{1}{\sqrt[3]{48}}$, then by \eqref{eq:mpp relation for g>1},  $deg\cdot\mpp=-\frac{2\pi}{3}\chi$. With \eqref{eq:mpp relation for g>1} and that $\frac{1}{2}deg\cdot \mpp+\mathfrak{a}>0$, we have $\mpp=\frac{2\pi}{3}$ and $deg=-\chi$. Here, $\ma$ can be any positive number.
\end{itemize}

Summarizing above one obtains that
\begin{itemize}
    \item If $deg\neq -\chi$, then $k\neq\frac{1}{\sqrt[3]{48}}$ and $\mpp,\ma$ are given by $\eqref{eq:decoupled p and a for g>1}$. Further computation gives
    \begin{align}
        e^w&={\left(-\chi(1 + 96k^3) ( \xi^2 + 12k^3 (1 + 2\xi) ) - 2deg\cdot ( 1152k^6 + \xi^3 - 6k^3 (1 + 2\xi)(1-8\xi+4\xi^2) )\right)} \label{eq:decoupled e^w Sigma_g}\\
        &\quad\cdot \frac{2\pi (1 - 2\xi)}{\operatorname{Vol}_{\ttg}(48k^3-1)(1 + 96k^3)}\notag, \\
        We^w&=\frac{24 k^3 \pi \left( -\chi(1 + 96k^3) - deg \left( -1 + 192 k^3 (1 - \xi) + 4 \xi \right) \right)}{\operatorname{Vol}_{\ttg}\left(48 k^3-1 \right) \left( 1 + 96 k^3 \right)}. \label{eq:decoupled We^w Sigma_g}
    \end{align}
    In the anti-self-dual case ($k=\pm\infty$), they simplify as
    \[
    e^{w}=\frac{\pi}{2\mathrm{Vol}_{\ttg}}\,(1-2\xi)\,\bigl(-\chi(1+2\xi)-2deg\bigr),
    \]
    \[
    W e^{w}=\frac{\pi}{2\,\mathrm{Vol}_{\ttg}}\Bigl(-\chi-2deg(1-\xi)\Bigr).
    \]
    \item If $deg=-\chi$, then $k=\frac{1}{\sqrt[3]{48}}, \mpp=\frac{2\pi}{3}$, and $\ma$ can take arbitrary positive value. Further computation gives
    \begin{align}
    e^w&=\frac{1}{\operatorname{Vol}_{\ttg}}(1-2\xi)\left(\mathfrak{a}(1+2\xi+4\xi^2)-\frac{2\pi\chi}{3}\xi(2+\xi+2\xi^2)\right),\label{eq:decoupled e^w Sigma_g degree -chi}\\
    We^w&=\frac{1}{\mathrm{Vol}_{\ttg}}(\mathfrak{a}-\frac{2}{3}\pi\chi\xi).\label{eq:decoupled We^w Sigma_g degree -chi}
  \end{align}
\end{itemize}

Again, in order for the decoupled solution to define a smooth metric, it is necessary to assume conditions \ref{C1}--\ref{C3}.

\begin{proposition}[$\Sigma_{\ttg}$ base, $\ttg>1$]\label{prop:range for k Sigma_g}
    The decoupled solution defines a smooth metric only if either
    \begin{itemize}
        \item $k\in  [-\infty,-\frac{1}{\sqrt[3]{96}})\cup (\frac{1}{\sqrt[3]{48}},\infty]$ and $deg<-\chi \frac{96k^3+1}{192k^3-1}$;
        \item $k=\frac{1}{\sqrt[3]{48}}$, $deg=-\chi$, $\mathfrak{a}>0$ is arbitrary;
        \item $k\in (\frac{1}{\sqrt[3]{192}},\frac{1}{\sqrt[3]{48}})$ and $deg>-\chi \frac{96k^3+1}{192k^3-1}$.
       
    \end{itemize}
\end{proposition}
\begin{proof}
  We omit the case $k=\pm\infty$. The case $k=\tfrac{1}{\sqrt[3]{48}}$ is also straightforward from \eqref{eq:decoupled e^w Sigma_g degree -chi}--\eqref{eq:decoupled We^w Sigma_g degree -chi}.

  Next, we assume $k\neq\frac{1}{\sqrt[3]{48}}$ and $k\neq\pm\infty$. Since $\mathfrak{p}>0$, this implies $k<-\frac{1}{\sqrt[3]{96}}$ or $k>0$. The condition \ref{C2} is equivalent to $\ma>0$ and $\frac{1}{2}deg\cdot\mpp+\ma>0$, where the latter is equivalent to  
  $$
  -\frac{\chi+deg}{1-\frac{1}{48k^3}}>0.
  $$

  \textbf{Case 1:} $k>\frac{1}{\sqrt[3]{48}}$ or $k<-\frac{1}{\sqrt[3]{96}}$.

  We have $deg<-\chi$. Since $\ma=-\frac{\chi(2+\frac{1}{48k^3})+deg(4-\frac{1}{48k^3})}{2(1-\frac{1}{48k^3})(2+\frac{1}{48k^3})}$, we know $\ma>0$ if and only if $deg<-\chi \frac{96k^3+1}{192k^3-1}$. In the following, we verify that this condition also implies \ref{C3}. Let  
  $$
  A(\xi):=-\chi(1 + 96k^3)\bigl( \xi^2 + 12k^3 (1 + 2\xi) \bigr) - 2\,deg\bigl( 1152k^6 + \xi^3 - 6k^3 (1 + 2\xi)(1-8\xi+4\xi^2) \bigr).
  $$
  The condition \ref{C3} is equivalent to $A(\xi)>0$ on $[0,\frac{1}{2}]$. Denote $p(\xi):=(1 + 96k^3)\bigl( \xi^2 + 12k^3 (1 + 2\xi) \bigr)$ and $q(\xi)=1152k^6 + \xi^3 - 6k^3 (1 + 2\xi)(1-8\xi+4\xi^2)$. Then  
  \[
  A(\xi)
  = -\chi\,p(\xi) - 2\,deg\,q(\xi)
  = -\chi\,p(\xi)\!\left(1 - 2\,\frac{deg}{-\chi}\,\frac{q(\xi)}{p(\xi)}\right).
  \]
  We claim that $p(\xi)>0$ on $[0, \tfrac{1}{2}]$. This is immediate when $k > \tfrac{1}{\sqrt[3]{48}}$. For $k < -\tfrac{1}{\sqrt[3]{96}}$, this holds because $1 + 96k^3 < 0$, and $\xi^2 + 12k^3(1 + 2\xi)$ is negative on $[0, \tfrac{1}{2}]$, since it is convex and negative at both endpoints.
  As for $1-2\frac{deg}{-\chi}\frac{q(\xi)}{p(\xi)}$, we have  
  $$
  \frac{d}{d\xi}\left(\frac{q(\xi)}{p(\xi)}\right)
  = -\,\frac{(48 k^3 - 1)(12 k^3 + \xi^3)\,(48 k^3 + \xi)}{(96 k^3 + 1)\,\bigl(12 k^3 + 24 k^3 \xi + \xi^2\bigr)^2}
  < 0.
  $$
  Thus, $\frac{q(\xi)}{p(\xi)}$ is decreasing, and  
  $\max_{\xi\in[0,\frac{1}{2}]}\frac{q(\xi)}{p(\xi)}=\frac{q(0)}{p(0)}=\frac{192k^3-1}{2(96k^3+1)}$,  
  $\min_{\xi\in[0,\frac{1}{2}]}\frac{q(\xi)}{p(\xi)}=\frac{q(\frac{1}{2})}{p(\frac{1}{2})}=\frac{1}{2}$.  
  Hence,  
  $$
  \min_{\xi\in[0,\frac12]}\!\left( 1 - 2\,\frac{deg}{-\chi}\,\frac{q(\xi)}{p(\xi)} \right)
  =
  \begin{cases}
    \displaystyle 1 - \frac{deg}{-\chi}\,\frac{192k^{3}-1}{96k^{3}+1}, & deg>0, \\[8pt]
    \displaystyle 1 - \frac{deg}{-\chi}, & deg\le 0,
  \end{cases}
  $$
  which is positive if and only if $deg<-\chi \frac{96k^3+1}{192k^3-1}$.

  \medskip

  \textbf{Case 2:} $0<k<\frac{1}{\sqrt[3]{48}}$. 

  Here $deg>-\chi$. As $\ma=-\frac{\chi(2+\frac{1}{48k^3})+deg(4-\frac{1}{48k^3})}{2(1-\frac{1}{48k^3})(2+\frac{1}{48k^3})}$, we know $\ma>0$ if and only if $k>\frac{1}{\sqrt[3]{192}}$ and $deg>-\chi \frac{96k^3+1}{192k^3-1}$. Next, we verify that this condition implies \ref{C3}. Denote $A(\xi), p(\xi), q(\xi)$ as in \textbf{Case 1}. The condition \ref{C3} is equivalent to $A(\xi)<0$ on $[0,\frac{1}{2}]$. In this case, $\frac{q(\xi)}{p(\xi)}$ is increasing; thus,  
  $\min_{\xi\in[0,\frac{1}{2}]}\frac{q(\xi)}{p(\xi)}=\frac{q(0)}{p(0)}=\frac{192k^3-1}{2(96k^3+1)}$, and  
  $$
  \max\limits_{\xi\in[0,\frac{1}{2}]}\!\left(1-2\frac{deg}{-\chi}\frac{q(\xi)}{p(\xi)}\!\right)
  =1-\frac{deg}{-\chi}\frac{192k^3-1}{96k^3+1},
  $$
  which is negative if and only if $k>\frac{1}{\sqrt[3]{192}}$ and $deg>-\chi\frac{96k^3+1}{192k^3-1}$.
\end{proof}
    It is worth noting that \eqref{eq:T2 decoupled period}–\eqref{eq:decoupled We^w T^2 degree 0} arise as the special case of \eqref{eq:decoupled p and a for g>1}–\eqref{eq:decoupled We^w Sigma_g degree -chi} obtained by setting $\chi=0$.

\subsection{Decoupled solutions with base \texorpdfstring{$S^2$}{}}
\label{subsec:S2 decouple}

In this situation, there are two more cases: the bolt case and the nut case. In the bolt case, the fixed point set of the Killing field is $\Xi\simeq S^2$, while in the nut case the fixed point set is an isolated point.

For the bolt case, the underlying manifold $M$ is the total space of a complex line bundle over $S^2$. The functions $e^w, We^w$ are still given by \eqref{eq:decoupled e^w Sigma_g}-\eqref{eq:decoupled We^w Sigma_g} when $deg\neq-\chi$, and \eqref{eq:decoupled e^w Sigma_g degree -chi}-\eqref{eq:decoupled We^w Sigma_g degree -chi} when $deg=-\chi$. Note that here $\chi=2$. Moreover, in the former case $\mpp,\ma$ are given by \eqref{eq:decoupled p and a for g>1}, while in the later case, $\mpp=\frac{2\pi}{3}$ and $\ma>0$. Again, they define a smooth metric on $M$ only if conditions \ref{C1}-\ref{C3} hold. Similar to Proposition \ref{prop:range for k Sigma_g}, we have the following 

\begin{proposition}[$S^2$ base, bolt case]\label{prop:range for k S2 bolt}
       The decoupled solution defines a smooth metric only if either
    \begin{itemize}
        \item $k\in  [-\infty,-\frac{1}{\sqrt[3]{96}})\cup (\frac{1}{\sqrt[3]{48}},\infty]$ and $deg<-\chi $;
        \item $k=\frac{1}{\sqrt[3]{48}}$, $deg=-\chi$, $\mathfrak{a}>\frac{\pi\chi}{3}$ is arbitrary;
        \item $k\in [\frac{1}{\sqrt[3]{192}},\frac{1}{\sqrt[3]{48}})$ and $deg>-\chi$;
        \item  $k\in (0,\frac{1}{\sqrt[3]{192}})$ and $-\chi<deg<-\chi \frac{96k^3+1}{192k^3-1}$.
    \end{itemize}
\end{proposition}

\begin{proof}
  We divide our discussion into the following cases.

  \medskip

  \textbf{Case 0:} $k=\frac{1}{\sqrt[3]{48}}$.

  Here $deg=-\chi$, and $e^w, We^w$ are given by \eqref{eq:decoupled e^w Sigma_g degree -chi}--\eqref{eq:decoupled We^w Sigma_g degree -chi}. Now \ref{C1} is equivalent to $\mathfrak{a}>0$, \ref{C2} is equivalent to $\mathfrak{a}>\frac{\pi\chi}{3}$, and \ref{C3} is equivalent to 
  $$
  \mathfrak{a}(1+2\xi+4\xi^2)-\tfrac{2\pi\chi}{3}\,\bigl(2\xi+\xi^2+2\xi^3\bigr)
  =\Bigl(\mathfrak{a}-\tfrac{\pi\chi}{3}\Bigr)(1+2\xi+4\xi^2)+\tfrac{\pi\chi}{3}(1-2\xi)(1+2\xi^2)>0.
  $$
  Hence, \ref{C1}--\ref{C3} together are equivalent to $\mathfrak{a}>\frac{\pi\chi}{3}$.

  \medskip

  Next, we assume $k\neq\frac{1}{\sqrt[3]{48}}$. Denote $A(\xi), p(\xi), q(\xi)$ as in Proposition \ref{prop:range for k Sigma_g}. Again, \ref{C1}--\ref{C2} are equivalent to $k<-\frac{1}{\sqrt[3]{96}}$ or $k>0$, together with $\ma>0$ and $-\frac{\chi+deg}{1-\frac{1}{48k^3}}>0$.

  \medskip

  \textbf{Case 1:} $k>\frac{1}{\sqrt[3]{48}}$ or $k<-\frac{1}{\sqrt[3]{96}}$.

  We have $deg<-\chi$. Since $\ma=-\frac{\chi(2+\frac{1}{48k^3})+deg(4-\frac{1}{48k^3})}{2(1-\frac{1}{48k^3})(2+\frac{1}{48k^3})}$, we know $\ma>0$ if and only if $deg<-\chi \frac{96k^3+1}{192k^3-1}$. One has $p(\xi)>0$ on $[0, \tfrac{1}{2}]$, 
  with the proof identical to that in Proposition~\ref{prop:range for k Sigma_g}. Now \ref{C3} is equivalent to $A(\xi)=-\chi p(\xi)\bigl(1+2\frac{deg}{\chi}\frac{q(\xi)}{p(\xi)}\bigr)>0$ on $[0,\frac12]$. 
  One again has $\frac{d}{d\xi}\frac{q(\xi)}{p(\xi)}<0$ on $[0,\frac12]$. 
  Hence, $A(\xi)>0$ is equivalent to
  $$
  \max\limits_{\xi\in[0,\frac{1}{2}]}\!\left(1+2\frac{deg}{\chi}\frac{q(\xi)}{p(\xi)}\!\right)
  =1+2\frac{deg}{\chi}\frac{q(\frac{1}{2})}{p(\frac{1}{2})}
  =1+\frac{deg}{\chi}<0.
  $$
  Because $1> \frac{96k^3+1}{192k^3-1}$, we conclude that condition \ref{C3} is equivalent to $deg<-\chi$.

  \medskip

  \textbf{Case 2:} $0<k<\frac{1}{\sqrt[3]{48}}$. 

  We have $deg>-\chi$. Since $\ma=-\frac{\chi(2+\frac{1}{48k^3})+deg(4-\frac{1}{48k^3})}{2(1-\frac{1}{48k^3})(2+\frac{1}{48k^3})}$, we know $\ma>0$ if and only if $deg>-\chi \frac{96k^3+1}{192k^3-1}$. As before, one can check $p(\xi)>0$ on $[0,\frac12]$. Now \ref{C3} is equivalent to $A(\xi)=-\chi p(\xi)\bigl(1+2\frac{deg}{\chi}\frac{q(\xi)}{p(\xi)}\bigr)<0$ on $[0,\frac12]$. 
  Again, we have $\frac{d}{d\xi}\frac{q(\xi)}{p(\xi)}>0$ on $[0,\frac12]$. Thus, $\frac{q(\xi)}{p(\xi)}$ is increasing, and 
  $\min_{\xi\in[0,\frac{1}{2}]}\frac{q(\xi)}{p(\xi)}=\frac{q(0)}{p(0)}=\frac{192k^3-1}{2(96k^3+1)}$, 
  $\max_{\xi\in[0,\frac{1}{2}]}\frac{q(\xi)}{p(\xi)}=\frac{g(\frac{1}{2})}{f(\frac{1}{2})}=\frac{1}{2}$. 
  Hence, $A(\xi)<0$ is equivalent to
  $$
  0<\min_{\xi\in[0,\frac12]}\!\left( 1 +2\,\frac{deg}{\chi}\,\frac{q(\xi)}{p(\xi)} \right)
  =
  \begin{cases}
    \displaystyle 1 + \frac{deg}{\chi}\,\frac{192k^{3}-1}{96k^{3}+1}, & deg>0, \\[8pt]
    \displaystyle 1 +\frac{deg}{\chi}, & deg\le 0.
  \end{cases}
  $$ 
  Divide into the two subcases $k\geq \frac{1}{\sqrt[3]{192}}$ and $0<k<\frac{1}{\sqrt[3]{192}}$. We conclude that if $k\geq \frac{1}{\sqrt[3]{192}}$, then condition \ref{C3} is equivalent to $deg>-\chi$. If $0<k<\frac{1}{\sqrt[3]{192}}$, then condition \ref{C3} is equivalent to $-\chi<deg<-\chi\frac{96k^3+1}{192k^3-1}$. This finishes the proof.
\end{proof}

\medskip

Now we consider the nut case. In this case, in order for the solutions to define a smooth metric away from the isolated point $\xi=\frac{1}{2}$, we need
\begin{enumerate}
    \item  $\mathfrak{p}>0, \ \mathfrak{a}>0$,
    \item  $We^w>0$ on $[0,\frac12)$, $We^w=0$ at $\xi=\frac{1}{2}$,
    \item  $e^w>0$ on $[0,\frac12)$, $e^w(\frac12)=0$,
\end{enumerate}
since as we approach $\xi=\frac{1}{2}$, the K\"ahler reduction shrinks to the nut point.
As $We^w=e^w=0$ at $\frac{1}{2}$, we have 
$$\frac{1}{2}deg \cdot\mathfrak{p}+\mathfrak{a}=0$$ 
and  \eqref{eq:e^w vanishes at 1/2}. Hence, $k\neq-\frac{1}{\sqrt[3]{96}}$ and 
$$\mathfrak{a}=\frac{48\pi\chi k^3}{96k^3+1}, \quad \mathfrak{p}=-\frac{2\mathfrak{a}}{deg}=\frac{1}{-deg}\frac{96\pi \chi k^3}{96k^3+1}.$$
The first condition implies $k\in  [-\infty,-\frac{1}{\sqrt[3]{96}})\cup (\frac{1}{\sqrt[3]{48}},\infty]$, and $deg<0$. Further computations gives 
$$e^w=\frac{(24k^3+\xi^2)\,(1-2\xi)^2}{96k^3+1},$$
$$We^w=\frac{24\,k^3}{96k^3+1}\,(1-2\xi),$$
$$W^{-1}=(1-2\xi)\!\left(1+\frac{\xi^2}{24k^3}\right).$$
Thus $\frac12$ is a simple zero of $We^w$, is a double zero of $e^w$, and is a simple zero of $W^{-1}$. As a consequence, we have
\begin{equation}\label{eq:Nut metric on B4}
    g
=\frac{1}{1-2\xi}\,\frac{24k^{3}}{24k^{3}+\xi^{2}}\,d\xi^{2}
+(1-2\xi)\!\left(\frac{24k^{3}+\xi^{2}}{24k^{3}}\,\eta^{2}
+\frac{24\,k^{3}}{96k^{3}+1}\,g_{S^{2}}\right).
\end{equation}
Denote $\beta=\frac{96k^3}{96k^3+1}>0$, $r=\sqrt{1-2\xi}\in (0,1]$. It follows that as $r\rightarrow 0^+$,
$$g\sim \beta(dr^2+r^2(\beta^{-2}\eta^2+\frac{1}{4}g_{S^2}))$$
Thus, for $g$ to define a smooth metric, it is necessary that $\frac{\mathfrak{p}}{\beta} = 2\pi$, that is, $deg = -1$. This agrees with the fact that, in the case of $B^4$, the underlying $\mathbb{S}^1$-bundle is the Hopf fibration.
 The metric \eqref{eq:Nut metric on B4} is relatively simple and one can check that it defines a smooth metric on $B^4$.

\begin{proposition}[$S^2$ base, nut case]\label{prop:range for k S2 nut}
    The decoupled solution defines a smooth metric, if and only if the base is $S^2$, $deg=-1$, $g$ is defined on $\overline{B^4}$  given by  the completion of \eqref{eq:Nut metric on B4}, and $k\in [-\infty,-\frac{1}{\sqrt[3]{96}})\cup (0,\infty]$.
\end{proposition}

\begin{remark}
In our case $g|_{\xi=0}=\eta^2+ \tfrac{\beta}{4}g_{S^2}$ can be written as $\beta^2 \sigma_1^2+\beta (\sigma_2^2+\sigma_3^2)$, where $\sigma_1,\sigma_2,\sigma_3$ are a left-invariant orthonormal coframe of $(S^3,g_{S^3})$, which is a Berger sphere with $\beta\in (0,\infty)$. When $k\in (-\infty,-\frac{1}{\sqrt[3]{96}})$, one has $\beta\in (1,\infty)$; when $k=\pm\infty$, $\beta=1$; and when $k\in (0,\infty)$, $\beta\in (0,1)$. This shows that the conformal infinity of every Berger sphere can fill in a conformally Kähler Poincaré–Einstein metric, which recovers Pedersen's construction \cite{Pedersen1986}. One can show that the Kähler metric $g$ defined by \eqref{eq:Nut metric on B4} is moreover self-dual under the complex orientation.
\end{remark}

\begin{remark}
    Without assuming $deg=-1$, for any $d:=-deg>0$, the metric $g$ defined by \eqref{eq:Nut metric on B4} is a K\"ahler orbifold metric on $\overline{B^4}/\mathbb{Z}_{d}$. Moreover, on the complex line bundle $\mathcal{O}(-2)$, the K\"ahler metric $g$ determined by \eqref{eq:decoupled e^w Sigma_g degree -chi}–\eqref{eq:decoupled We^w Sigma_g degree -chi} converges smoothly, away from $\xi=\frac{1}{2}$, to the orbifold K\"ahler metric \eqref{eq:Nut metric on B4} with $k=\tfrac{1}{\sqrt[3]{48}}$ on $\overline{B^4}/\mathbb{Z}_2$, as $\mathfrak{a}\to \tfrac{2\pi}{3}^+$. Thus we obtain a sequence of Poincar\'e–Einstein metrics on $\mathcal{O}(-2)$ converging, in the pointed Gromov–Hausdorff sense, to a non-hyperbolic Poincar\'e–Einstein orbifold on $B^4/\mathbb{Z}_2$.
\end{remark}

\subsection{Regularity and conclusion}
\label{subsec:conclusion}
We first demonstrate that all the necessary conditions appearing in the bolt cases Proposition \ref{prop:range for k T2}-\ref{prop:range for k S2 bolt} are also sufficient. That is, in the bolt case, the decoupled solutions \eqref{eq:decoupled solutions w}–\eqref{eq:decoupled solutions W} satisfying $\frac{\mpp}{2A_0}=2\pi$ and conditions \ref{C1}–\ref{C3} give rise to smooth conformally Kähler Poincaré–Einstein metrics. To see this, we work in a local holomorphic coordinate $x+iy$ for the Riemann surface $\Sigma$, where the $\mathbb{S}^1$-bundle is trivialized. Fix the gauge by writing the 1-form $\eta=d\theta+\mathcal{X}dx+\mathcal{Y}dy$, with $\theta$ parameterizing the action induced by $\mathcal{K}$ and functions $\mathcal{X},\mathcal{Y}$ determined by \eqref{eq:d eta}. Introduce the coordinate change $\tau:=\sqrt{\mpp/\pi}(\frac{1}{2}-\xi)^{\frac{1}{2}}$, then set $x_1:=\tau\cos\theta,x_2:=\tau\sin\theta,x_3:=x,x_4:=y$.  It is direct to check that the Einstein metric $h$ under $(x_1,x_2,x_3,x_4)$ is Lipschitz around $\xi=1/2$. By standard elliptic regularity for Einstein metrics (see Proposition 4.15 of \cite{Li_Sun_2025} for example), it follows that there is a coordinate such that the Einstein metric $h$ is smooth and our conclusion follows.

Summarizing all bolt cases, the necessary and sufficient conditions can be formulated as the following \textit{admissibility} condition. The definition is separated into the case of genus greater than zero and the case of genus zero.

\begin{definition}\label{def:admissible five tuple}
    A five tuple $(deg,\chi,k,\mathfrak{a},\mathfrak{p})$, where $deg\in\mathbb{Z}$ and $\chi=2-2\ttg$ with $\ttg\in\mathbb{Z}_{>0}$, is said to be {admissible} if one of the following holds:
\begin{itemize}
  \item $deg \neq -\chi$, $k \neq \tfrac{1}{\sqrt[3]{48}}$, and  
  \[
  \mathfrak{p} = \frac{96\pi k^3}{96k^3+1}, 
  \qquad 
  \mathfrak{a} = -\frac{\pi(\chi+deg)}{2\left(1-\tfrac{1}{48k^3}\right)} 
                 - \frac{\pi \,deg}{2+\tfrac{1}{48k^3}}.
  \]  
  Moreover, either  
  \[
  k \in         \Bigl[-\infty,-\tfrac{1}{\sqrt[3]{96}}\Bigr)\cup \Bigl(\tfrac{1}{\sqrt[3]{48}},\infty\Bigr],
  \qquad 
  deg < -\chi \,\frac{96k^3+1}{192k^3-1},
  \]  
  or  
  \[
  k \in \Bigl(\tfrac{1}{\sqrt[3]{192}},\tfrac{1}{\sqrt[3]{48}}\Bigr),
  \qquad 
  deg > -\chi \,\frac{96k^3+1}{192k^3-1}.
  \]

  \item $deg = -\chi$, $k = \tfrac{1}{\sqrt[3]{48}}$, and  
  \[
  \mathfrak{p} = \tfrac{2\pi}{3}, 
  \qquad 
  \mathfrak{a} > 0 \ \text{ arbitrary}.
  \]
\end{itemize}
\end{definition}

\begin{definition}\label{def:admissible five tuple S2 case}
    A five tuple $(deg,\chi,k,\mathfrak{a},\mathfrak{p})$, where $deg\in\mathbb{Z}$ and $\chi=2$, is said to be {admissible} if one of the following holds:
\begin{itemize}
  \item $deg \neq -\chi$, $k \neq \tfrac{1}{\sqrt[3]{48}}$, and  
  \[
  \mathfrak{p} = \frac{96\pi k^3}{96k^3+1}, 
  \qquad 
  \mathfrak{a} = -\frac{\pi(\chi+deg)}{2\left(1-\tfrac{1}{48k^3}\right)} 
                 - \frac{\pi \,deg}{2+\tfrac{1}{48k^3}}.
  \]  
  Moreover, either  
  \[
  k \in         \Bigl[-\infty,-\tfrac{1}{\sqrt[3]{96}}\Bigr)\cup \Bigl(\tfrac{1}{\sqrt[3]{48}},\infty\Bigr],
  \qquad 
  deg < -\chi, 
  \]  
  or  
  \[
  k \in \Bigl[\tfrac{1}{\sqrt[3]{192}},\tfrac{1}{\sqrt[3]{48}}\Bigr),
  \qquad 
  deg > -\chi,
  \]
  or
   \[
  k \in \Bigl(0,\tfrac{1}{\sqrt[3]{192}}\Bigr),
  \qquad 
   -\chi<deg<-\chi \frac{96k^3+1}{192k^3-1}.
  \]

  \item $deg = -\chi$, $k = \tfrac{1}{\sqrt[3]{48}}$, and  
  \[
  \mathfrak{p} = \tfrac{2\pi}{3}, 
  \qquad 
  \mathfrak{a} >\tfrac{\pi\chi}{3}  \ \text{ arbitrary}.
  \]
\end{itemize}
\end{definition}

\begin{theorem}\label{thm:decoupled summary for bolt}
    Let $\Sigma$ be a closed Riemann surface with Euler characteristic $\chi$ and $deg\in\mathbb{Z}$. Then for each admissible $(deg,\chi,k,\mathfrak{a},\mathfrak{p})$, there is a cohomogeneity-one Poincar\'e–Einstein metric of regular conformally K\"ahler geometry on the total space of complex line bundle over $\Sigma$ of degree $deg$, explicitly given by
    \begin{align}
        h&=\xi^{-2}\bigl(W\,d\xi^2+W^{-1}\eta^2+We^w g_{\Sigma}\bigr), \label{eq:PE metric general} \\
        e^{w}&=\frac{1}{\mathrm{Vol}_{\Sigma}}\left(-\frac{{deg}\cdot\mathfrak{p}}{12k^3}\xi^4-\frac{\mathfrak{a}}{6k^3}\xi^3+2\pi\chi(\Sigma)\xi^2+2({deg}\cdot\mathfrak{p})\xi+\mathfrak{a}\right),\\
        W&=\frac{({deg}\cdot\mathfrak{p})\xi+\mathfrak{a}}{-\frac{{deg}\cdot\mathfrak{p}}{12k^3}\xi^4-\frac{\mathfrak{a}}{6k^3}\xi^3+2\pi\chi(\Sigma)\xi^2+2({deg}\cdot\mathfrak{p})\xi+\mathfrak{a}},
    \end{align} 
    with $\xi\in[0,\frac{1}{2})$ and $\eta$ determined by \eqref{eq:d eta final}.
\end{theorem}

As can be seen from the definition, fixing the topological parameters $\chi$ and $deg$, the set of admissible tuples is parameterized by $k$ if $deg\neq\chi$ and is parameterized by $\ma$ if $deg=-\chi$. Hence, we obtain in fact a one-parameter family of Poincaré–Einstein metrics on the total space of any complex line bundle over any closed Riemann surface.
Note that $g:=\xi^2h$ is a smooth K\"ahler manifold with boundary whose scalar curvature ranges between $0$ and $\tfrac{1}{2k^3}$, vanishes on the boundary, and attains its maximum at the bolt $\xi=\frac{1}{2}$. As we mentioned in the introduction, such K\"ahler metrics are actually extremal K\"ahler.

The above computation remains valid, in the averaged sense, for general regular conformally K\"ahler Poincar\'e-Einstein metrics defined either over the total space of complex line bundles over $\Sigma$, or over $B^4$. For such a Poincar\'e-Einstein metric to be smooth, without introducing cone angles or orbifold singularities, the tuple $(deg,\chi,k,\ma,\mpp)$ must be admissible.

\section{The Dirichlet boundary value problem}
\label{sec:Dirichlet boundary value problem}

Fix an oriented smooth surface $(\Sigma,g^\natural)$ of genus $\ttg \geq 1$ and let $M$ be the smooth $4$-manifold underlying the total space of a complex line bundle $L \to \Sigma$ of degree $deg$. Suppose $(deg,\chi,k,\ma,\mpp)$ is admissible, with $\chi=\chi(\Sigma)$ and $\ma$ being the area of $g^\natural$. We then consider the following fill-in problem.
\begin{question}\label{ques: two dimensional fill in}
    When do $(\Sigma,g^{\natural})$ and the admissible tuple $(deg,\chi,k,\ma,\mpp)$ fill in a regular conformally K\"ahler Poincar\'e-Einstein metric on $M$?
\end{question}
By saying that they fill in a regular conformally K\"ahler Poincar\'e-Einstein metric on $M$, we  mean that there exists a conformally K\"ahler Poincar\'e-Einstein metric $(M,h,g)$ such that the 2-dimensional infinity of $h$ is precisely $g^\natural$, $\mathfrak{p}$ is the length of the $\mathbb{S}^1$ fibers of $g^\flat$, $\ma$ is the area of $(\Sigma,g^\natural)$, and $k$ is the number satisfying $k\cdot (2\sqrt{6}\max|\mathscr{W}_h^+|_h)^{1/3}=\frac{1}{2}\sign(s_g)$ when $h$ is not anti–self–dual, with $k=\pm\infty$ when $h$ is anti–self–dual. Writing $g^{\natural} = e^\varphi g_{\Sigma}$, where $g_{\Sigma}$ has constant Gauss curvature $K_{\Sigma}=0$ or $-1$ (with volume $4\pi^2$ when $\ttg=1$), the geometric fill-in problem reduces, at the PDE level, to the following Dirichlet boundary value problem.

\begin{question}\label{ques:pde fill in}
    Given the admissible tuple $(deg,\chi,k,\ma,\mpp)$, does the following Dirichlet boundary value problem 
    \begin{equation}\label{eq:pde fill in}
        \left\{
            \begin{aligned}
                &(e^w)_{\xi\xi}+\Delta_\Sigma w-2K_\Sigma=-\xi e^w\frac{12-6\xi \partial_\xi w}{12k^3+\xi^3},\\
                &w|_{\xi=0}=\varphi, \quad \int_{\Sigma}e^\varphi d\operatorname{vol}_{\Sigma}=\ma
            \end{aligned}
        \right.
    \end{equation}
    have a unique solution $w$ such that $w-\log(\frac{1}{2}-\xi)\in C^{2,\alpha}([0,\frac{1}{2}]\times \Sigma)$?
\end{question}

The main theorem of this section is the following existence and uniqueness for Question \ref{ques:pde fill in}.

\begin{theorem}\label{thm:existence and uniqueness pde}
    There exists a unique solution $w$ to Question \ref{ques:pde fill in}.
\end{theorem}

In order to relate to Question \ref{ques: two dimensional fill in}, one still needs to prove that solutions $w$ to Question \ref{ques:pde fill in} have geometric meaning, in the sense that they correspond to smooth Poincar\'e-Einstein metrics via \eqref{eq:g final}-\eqref{eq:d eta final}. In particular, one needs to show $W=\frac{12-6\xi w_{\xi}}{12+\xi^3/k^3}>0$, which is not clear at all at the level of PDEs. This will be proved in Section \ref{subsec:positivity of W} by a geometric continuity argument. Hence, together with regularity theory of Einstein metrics, our Theorem \ref{thm:existence and uniqueness pde} also settles down the geometric fill-in problem Question \ref{ques: two dimensional fill in}, showing both existence and uniqueness for conformally K\"ahler PE fill-ins.

 Denote $\overline{w}(\xi)=\log(\frac{1}{\operatorname{Vol}_{\Sigma}}\int_{\{\xi\}\times\Sigma}  e^w d\operatorname{vol}_{\Sigma})$, which is the corresponding decoupled solution. In the following we always assume the genus $\ttg\geq1$.

\subsection{\texorpdfstring{$C^0$ bound assuming $\ttg \ge 1$}{C0 bound when ttg >= 1}}

Recall that $e^{\overline{w}}$ is a polynomial in $\xi$, positive on $[0,\frac{1}{2})$ and has a simple zero at $\xi=\frac{1}{2}$. It moreover satisfies the ODE  $$(e^{\overline{w}})_{\xi\xi}+\xi e^{\overline{w}} \frac{12-6\xi \partial_{\xi}\overline w}{12k^3+\xi^3}=2K_{\Sigma}.$$
Also, $\int_{\{\xi\}\times\Sigma}e^w=\int_{\{\xi\}\times\Sigma}e^{\overline{w}}$ for any $\xi\in[0,\frac{1}{2})$. We set $u=w-\overline{w}$, $a(\xi)=\frac{-6\xi^2}{12k^3+\xi^3}$, $b(\xi)=\frac{12\xi}{12k^3+\xi^3}$, and $\psi(\xi)=e^{\overline{w}}/(\frac12-\xi)$, then $u$ satisfy the following PDE
\begin{equation}\label{eq:PDE in u}
\Delta_{\Sigma} u+e^{\overline{w}}(e^u)_{\xi\xi}+2(e^u)_{\xi}(e^{\overline{w}})_{\xi}+a(\xi)e^{\overline{w}}(e^u)_{\xi}+2K_{\Sigma}(e^u-1)=0.
\end{equation}
We have the following boundary condition along $\{0\}\times\Sigma$
\begin{equation}\label{eq:bc_PDE in u}
    u|_{\xi=0}=\varphi-\overline{w}|_{\xi=0}.
\end{equation}
In the following, we denote $u^+:=\max\{u,0\}$ and $u^-:=\max\{-u,0\}$.

\begin{lemma}\label{lem:maximum principle for u}
     Suppose $u\in C^2([0,\frac{1}{2}]\times\Sigma)$ and solves \eqref{eq:PDE in u},  then
     $$\sup_{[0,\frac{1}{2}]\times\Sigma} u\leq \sup_{\{0\}\times\Sigma}u^+, \quad\inf_{[0,\frac{1}{2}]\times\Sigma} u\geq \inf_{\{0\}\times\Sigma}-u^-.$$
\end{lemma}
\begin{proof}
 We shall prove the inequality for the supremum only, since the case of the infimum is similar. For any $\delta>0$, denote $\overline{w}_{\delta}$ the solution to the ODE 
  $$
  (e^{\overline{w}_\delta})_{\xi\xi}+\xi e^{\overline{w}_\delta} \frac{12-6\xi \partial_{\xi}{\overline w}_{\delta}}{12k^3+\xi^3}=2K_{\Sigma}-\delta
  $$
  with $\lim_{\xi\rightarrow{\frac{1}{2}}^-}\overline{w}_\delta(\xi)=-\infty$, and $\overline{w}_{\delta}-\overline{w}\rightarrow 0$ in $C^0([0,\tfrac{1}{2}])$ as $\delta\to0$, which is clear since one can write down $\overline{w}_\delta$ and $\overline{w}$ explicitly. Set $u_{\delta}=w-\overline{w}_{\delta}$. Then $u_\delta$ satisfies the PDE
  $$
  \Delta_{\Sigma} {u_\delta}+e^{{\overline{w}_{\delta}}}(e^{u_\delta})_{\xi\xi}+2(e^{u_\delta})_{\xi}(e^{{\overline{w}_{\delta}}})_{\xi}+a(\xi)e^{{\overline{w}_{\delta}}}(e^{u_\delta})_{\xi}+2K_{\Sigma}(e^{u_\delta}-1)-\delta e^{u_{\delta}}=0.
  $$
  We shall prove that $u_\delta$ cannot achieve a non-negative maximum in $(0,\frac{1}{2}]\times\Sigma$. Proceed with contradiction. Suppose $u_\delta$ achieves a non-negative maximum at $q_0\in (0,\tfrac{1}{2}]\times \Sigma$.
  
  \textbf{Case 1.} If $q_0\in (0,\tfrac{1}{2})\times\Sigma$, then at $q_0$ 
  $$
  (u_\delta)_{\xi}=0,\quad (e^{u_\delta})_{\xi\xi}\leq 0,\quad \Delta_{\Sigma}u_\delta\leq 0.
  $$
  
  Hence at $q_0$ we get a contradiction:
  $$
  0\leq 2K_{\Sigma}(e^{u_\delta}-1)-\delta e^{u_\delta}<0.
  $$
    
  \textbf{Case 2.} If $q_0\in\{\tfrac{1}{2}\}\times\Sigma$, then at $q_0$
  $$
  (u_\delta)_\xi\geq 0,\quad \Delta_\Sigma u_\delta\leq0.
  $$
  
  At $q_0$, we have $e^{{\overline{w}_{\delta}}}=0$ and $(e^{{\overline{w}_{\delta}}})_{\xi}<0$. Thus at $q_0$ we get a contradiction again:
  $$
  0=\Delta_{\Sigma}u_\delta+2(e^{u_\delta})_{\xi}(e^{{\overline{w}_{\delta}}})_{\xi}+2K_{\Sigma}(e^{u_\delta}-1)-\delta e^{u_\delta}\leq 2K_{\Sigma}(e^{u_\delta}-1)-\delta e^{u_\delta}.
  $$
  Hence 
  $
  \sup_{[0,\tfrac{1}{2}]\times\Sigma} u_\delta \leq \sup_{\{0\}\times\Sigma} (u_\delta)^+ .
  $
  Letting $\delta\rightarrow 0$, we obtain the desired conclusion. 
\end{proof}

\subsection{\texorpdfstring{Higher order estimates when $\ttg \geq 1$}{Higher order estimates when ttg >= 1}}

\begin{theorem}\label{thm:higher_regularity_estimates}
Suppose $u\in C^2([0,\frac{1}{2}]\times \Sigma)$ and solves \eqref{eq:PDE in u}. Fix $m\geq 1,\,0<\alpha<1$. If the boundary value $u|_{\{0\}\times\Sigma}\in C^{m,\alpha}(\{0\}\times\Sigma)$, then there exists a constant $C$ depending on $m,\alpha,\,|u|_{C^{m,\alpha}(\{0\}\times\Sigma)}$ such that
$$
|u|_{C^{m,\alpha}([0,\tfrac{1}{2}]\times\Sigma)}\leq C.
$$
Moreover, for any $n\in\mathbb{N}$, there exist constants $C_n$ depending on $n, |u|_{C^{0}(\{0\}\times\Sigma)}$ such that 
$$|u|_{C^n([\frac{1}{4},\frac{1}{2}]\times\Sigma)}\leq C_n.$$
\end{theorem}

\begin{proof}
Our proof breaks into the following steps.

\medskip\noindent
\textbf{Step 1: Lift to $\mathbb{R}^4\times\Sigma$.}
Let $\mathbf{z}=(z_{1},\dots,z_{4})\in\mathbb{R}^4$ and set
$
\xi = \tfrac12 - \tfrac14 |\mathbf{z}|^{2}.
$
Define the lifted function
\begin{equation}\label{eq: lift_function_definition}
\mathbf{u}(\mathbf{z},x) := u\!\left( \tfrac12 - \tfrac14 |\mathbf{z}|^{2},\, x \right),
\quad (\mathbf{z},x)\in \Omega := \{\,|\mathbf{z}|<\sqrt{2}\,\} \times \Sigma.
\end{equation}
A direct computation gives
$$
\partial_{z_i} \mathbf{u} = -\tfrac12 u_{\xi} z_{i},\quad
\partial^2_{z_i z_i} \mathbf{u} = \tfrac14 z_{i}^{2} u_{\xi\xi} - \tfrac12 u_{\xi},
$$
$$
|\nabla_{\mathbf{z}}\mathbf{u}|^{2} = \bigl(\tfrac12 - \xi \bigr) u_{\xi}^{2}, \quad
\Delta_{\mathbf{z}}\mathbf{u} = \bigl( \tfrac12 - \xi \bigr) u_{\xi\xi} - 2 u_{\xi}, \quad
\sum_{i=1}^4 -\tfrac12 z_{i} \partial_{z_i} \mathbf{u} = \bigl( \tfrac12 - \xi \bigr) u_{\xi}.
$$
Substituting these into \eqref{eq:PDE in u}, the equation for $\mathbf{u}$ on $\Omega$ becomes
\begin{equation}\label{eq:lifted equation for regularity}
    \Delta_x \mathbf{u} + \psi\Delta_{\mathbf{z}}e^{\mathbf{u}} - \tfrac{1}{2}\big(2\psi_{\xi} +  a \psi\big) e^{ \mathbf{u}}\, \mathbf{z}\cdot\nabla_{\mathbf{z}}\mathbf{u} + 2 K_{\Sigma} (e^{\mathbf{u}} - 1) = 0.
\end{equation}
By Lemma \ref{lem:maximum principle for u}, we know $e^{\mathbf{u}}$ is bounded from above and below by positive constants. It is also direct to apply maximum principle on \eqref{eq:lifted equation for regularity} to see that 
$$\sup\limits_{\Omega}\mathbf{u}\leq \sup\limits_{\partial\Omega} \mathbf{u}^+, \quad \inf\limits_{\Omega}\mathbf{u}\geq \inf\limits_{\partial\Omega} -\mathbf{u}^-.$$ 
The boundary $\{\xi=0\}\times\Sigma$ in $[0,\tfrac12]\times\Sigma$ corresponds to 
$\Sigma\times S^{3}(\sqrt{2}) \subset \partial\Omega$, and
$$
c_{1} |u|_{C^{m,\alpha}(\partial\Omega)} \le |\mathbf{u}|_{C^{m,\alpha}(\partial\Omega)} \le c_{2} |u|_{C^{m,\alpha}(\partial\Omega)},
$$
with constants $c_{1},c_{2}$ depending only on the change of variables.

\medskip\noindent
\textbf{Step 2: Regularity for $\mathbf{u}$.} 
Written in divergence form, \eqref{eq:lifted equation for regularity} becomes
\begin{equation}\label{eq:lifted equation divergence form}
    \Delta_x \mathbf{u} 
+ \operatorname{div}_{\mathbf{z}}\!\big(\psi e^{\mathbf{u}} \nabla_{\mathbf{z}} \mathbf{u}\big)
- e^{\mathbf{u}} \nabla_{\mathbf{z}} \psi \cdot \nabla_{\mathbf{z}} \mathbf{u}
- \tfrac{1}{2}\big(2\psi_{\xi} +  a \psi\big) e^{\mathbf{u}}\, \mathbf{z}\cdot\nabla_{\mathbf{z}}\mathbf{u}
+ 2 K_{\Sigma} \big(e^{\mathbf{u}} - 1\big) = 0.
\end{equation}
Denote $\Omega' := \{\,|\mathbf{z}|<1\,\} \times \Sigma\Subset\Omega$. We may freeze the $C^0$ function $e^{\mathbf{u}}$ in \eqref{eq:lifted equation divergence form}.
By the De Giorgi–Nash-Moser interior and boundary estimates Theorem 8.32 and Theorem 8.33 in \cite{GilbargTrudinger}, applied to the uniformly elliptic divergence-form PDE \eqref{eq:lifted equation divergence form}, we have
$$
|\mathbf{u}|_{C^{1,\alpha}(\overline{\Omega'})}
\le C\bigl( 1 + \|\mathbf{u}\|_{C^0(\overline{\Omega})}  \bigr)
\le C,
$$
$$
|\mathbf{u}|_{C^{1,\alpha}(\overline{\Omega})}
\le C\bigl( 1 + \|\mathbf{u}\|_{C^0(\overline{\Omega})} + |\mathbf{u}|_{C^{1,\alpha}(\partial\Omega)} \bigr)
\le C.
$$
Then, by interior and boundary Schauder estimates, we may use the nonlinearity of the PDE \eqref{eq:lifted equation for regularity} to bootstrap to a global $C^{m,\alpha}(\overline{\Omega})$ estimate, depending on $|\mathbf{u}|_{C^0(\overline{\Omega})},|\mathbf{u}|_{C^{m,\alpha}(\partial\Omega)}$, and an arbitrary $C^{n}(\overline{\Omega'})$ interior estimate, which depends only on $n,|\mathbf{u}|_{C^0(\overline{\Omega})}$.

\medskip\noindent
\textbf{Step 3: Reduction to $u$ and regularity near $\xi=\frac{1}{2}$.}
 By the following reduction lemma, for any $n\in\mathbb{N}$, we have $|u|_{C^{n}({[\frac{1}{4},\frac{1}{2}]\times\Sigma})} \le C_n\,|\mathbf{u}|_{C^{2n}(\overline{\Omega'})} \le C_n$. Away from the locus $\mathbf{z}=0$, the Jacobian matrix of the change of variables is uniformly controlled, so $$
 |u|_{C^{m,\alpha}([0,\frac{3}{8}]\times\Sigma)} \le C|\mathbf{u}|_{C^{m,\alpha}(\overline{\Omega})} \le C.
$$ It follows that $|u|_{C^{m,\alpha}([0,\frac{1}{2}]\times\Sigma)}\leq C$.
\end{proof}

\begin{lemma}\label{lem:half regularity}
Let $f:[0,R^2]\to\mathbb{R}$ and define $F:\overline{B_R(0)}\subset\mathbb{R}^n\to\mathbb{R}$ by $F(v)=f(|v|^2)$. 
Assume $F\in C^{2m}(\overline{B_R(0)})$ for some $m\in\mathbb{N}$. Then $f\in C^{m}([0,R^2])$, and 
$
|f|_{C^{m}([0,R^2])}\leq C\,|F|_{C^{2m}(\overline{B_R(0)})}.
$
\end{lemma}

\begin{proof}
Assume $m\ge1$. Fix any $e\in S^{n-1}$ and define $g(t):=F(te)=f(t^2)$, which is an even $C^{2m}$ function on $[-R,R]$. 
Thus, all odd derivatives of $g$ vanish at $0$. 
By subtracting the $2m$-th Taylor polynomial centered at $0$, we may assume 
$g(0)=\cdots=g^{(2m)}(0)=0$. 
Applying the integral remainder theorem to $g(t)$ gives 
$f(s)=\frac{1}{(2m-1)!}\int_{0}^{\sqrt{s}}\bigl(\sqrt{s}-\tau\bigr)^{2m-1}g^{(2m)}(\tau)\,d\tau$, with $s=t^2$. By induction on $q$, one shows that for any $1\le q\le m$, 
$f^{(q)}(0)=0$, 
$f^{(q)}(s)=\sum_{j=1}^{q}C_{q,j}\int_{0}^{\sqrt{s}} s^{\frac{j}{2}-q}\bigl(\sqrt{s}-\tau\bigr)^{2m-1-j}g^{(2m)}(\tau)\,d\tau$ when $s>0$, 
$\lim\limits_{s\to0^+}f^{(q)}(s)=f^{(q)}(0)$, and
$$
|f^{(q)}(s)|
\le \sum_{j=1}^{q} |C_{q,j}|\sqrt{s}\cdot s^{\frac{j}{2}-q}\cdot \sqrt{s}^{\,2m-1-j}\cdot \sup\limits_{\tau \in [-R,R]}|g^{(2m)}(\tau)|
= \sum_{j=1}^{q}|C_{q,j}|\,s^{m-q}\cdot \sup\limits_{\tau\in [-R,R]}|g^{(2m)}(\tau)|.
$$
The conclusion follows.
\end{proof}

\medskip

\subsection{\texorpdfstring{Existence theory when $\ttg \geq 1$}{Existence theory when ttg >= 1}}

\begin{theorem}
Fix any $m\geq 2,\alpha\in (0,1)$. For any $\varphi\in C^{m,\alpha}(\{0\}\times\Sigma)$, there exists a unique $u\in C^{m,\alpha}([0,\tfrac{1}{2}]\times\Sigma)\cap C^{\infty}([\tfrac{1}{4},\tfrac{1}{2}]\times\Sigma)$ satisfying \eqref{eq:PDE in u} and \eqref{eq:bc_PDE in u}.
\end{theorem}

We will prove a stronger version for the lifted equation. Let 
$\Omega := \{\,\mathbf{z}\in\mathbb{R}^4  \;\big|\; |\mathbf{z}|<\sqrt{2}\,\}\times \Sigma$, 
as in Theorem~\ref{thm:higher_regularity_estimates}. For any function defined on $[0,\tfrac{1}{2})\times\Sigma$, we define its lifting to $\Omega$ as before; in particular, this applies to the functions $a(\xi) = -\tfrac{6\xi^{2}}{12k^{3}+\xi^{3}}$ and $\psi(\xi) = {e^{\overline{w}}}/{(\tfrac{1}{2}-\xi)}$.

\begin{theorem}\label{thm:TypeII_lifted_solution_moduli_space}
For any $\boldsymbol{\varphi}\in C^{2,\alpha}(\overline{\Omega})$, there exists a unique solution 
$\mathbf{u} \in C^{2,\alpha}(\overline{\Omega})$ to the boundary value problem
\begin{align}
\Delta_x \mathbf{u} + \psi\Delta_{\mathbf{z}}e^{\mathbf{u}} - \tfrac{1}{2}\big(2\psi_{\xi} +  a \psi\big) e^{ \mathbf{u}}\, \mathbf{z}\cdot\nabla_{\mathbf{z}}\mathbf{u} + 2 K_{\Sigma} (e^{\mathbf{u}} - 1) &= 0, && \text{in } \Omega, 
\label{eq:TypeII_lifted_pde} \\
\mathbf{u} &= \boldsymbol{\varphi}, && \text{on } \partial\Omega.
\label{eq:bc_TypeII_lifted_pde}
\end{align}
Moreover, the moduli space of solutions
\begin{equation}\label{eq:lifted equation moduli space definition}
\mathcal{M} = \Bigl\{\,\mathbf{u} \in C^{2,\alpha}(\overline{\Omega}) \;\big|\;
\mathbf{u} \text{ solves \eqref{eq:TypeII_lifted_pde}}\,\Bigr\}
\end{equation}
is a smooth Banach manifold, and the boundary restriction map
\begin{equation*}
\mathcal{M} \longrightarrow C^{2,\alpha}(\partial\Omega), 
\qquad \mathbf{u} \mapsto \mathbf{u}|_{\partial\Omega},
\end{equation*}
is a diffeomorphism.
\end{theorem}

\begin{proof}
Define 
$\Phi: C^{2,\alpha}(\overline{\Omega}) \to C^{\alpha}(\overline{\Omega}) \times C^{2,\alpha}(\partial\Omega)$ 
by 
$$\Phi(\mathbf{u}) = (\Delta_x \mathbf{u} + \psi\Delta_{\mathbf{z}}e^{\mathbf{u}} - \tfrac{1}{2}\big(2\psi_{\xi} +  a \psi\big) e^{ \mathbf{u}}\, \mathbf{z}\cdot\nabla_{\mathbf{z}}\mathbf{u} + 2 K_{\Sigma} (e^{\mathbf{u}} - 1),\; \mathbf{u}|_{\partial\Omega}).$$  
This is a smooth map between Banach spaces. Its derivative at $\mathbf{u}$ is 
$D\Phi_{\mathbf{u}}(\mathbf{v}) = (L_{\mathbf{u}} \mathbf{v},\; \mathbf{v}|_{\partial\Omega})$, 
where \begin{equation}\label{eq:linearized operator general case}
\begin{aligned}
L_{\mathbf{u}}\mathbf{v}
&= \Delta_x \mathbf{v} 
   + \psi \Delta_{\mathbf{z}}\!\big(e^{\mathbf{u}}\mathbf{v}\big)
   - \tfrac{1}{2}\big(2\psi_\xi + a \psi\big)\mathbf{z}\cdot\nabla_{\mathbf{z}}(e^{\mathbf{u}}\mathbf{v})
   + 2K_{\Sigma}e^{\mathbf{u}}\mathbf{v}.
\end{aligned}
\end{equation}

\textbf{Step 1. Fredholm property.}  

We claim that for any $\mathbf{u} \in C^{2,\alpha}(\overline{\Omega})$, the operator $D\Phi_{\mathbf{u}} : C^{2,\alpha}(\overline{\Omega}) \to C^{\alpha}(\overline{\Omega}) \times C^{2,\alpha}(\partial\Omega)$   is Fredholm of index $0$. To see this, observe that $L_{\mathbf{u}}$ is homotopic to the Laplacian $\Delta_x+\Delta_{\mathbf{z}}$ via the family $(1-\sigma)L_{\mathbf{u}} + \sigma(\Delta_x+\Delta_{\mathbf{z}})$ with $\sigma \in [0,1]$. Since the Fredholm index is invariant under continuous homotopies of elliptic operators, $D\Phi_{\mathbf{u}}$ has the same index as $(\Delta_x+\Delta_{\mathbf{z}},\; \cdot|_{\partial\Omega})$, which is $0$.

\medskip

\textbf{Step 2. Invertibility of the linearized operator.}  

In the following Lemma \ref{lem:no cokernel for linearized operator when genus greater than 1},  we will show that for any $\mathbf{u}\in C^{2,\alpha}(\overline{\Omega})$, $D\Phi_{\mathbf{u}}$ is surjective. This is equivalent to say
the formal adjoint operator $L_{\mathbf{u}}^*$ of  $L_{\mathbf{u}}$ has no kernel with vanishing Dirichlet boundary data. Granting Lemma \ref{lem:no cokernel for linearized operator when genus greater than 1}, since $D\Phi_{\mathbf{u}}$ is of Fredholm index 0, we know $D\Phi_{\mathbf{u}}$ has no kernel as well. 
This implies the inequality  
$$|\mathbf{v}|_{C^{2,\alpha}(\overline{\Omega})} \leq C \big( |L_{\mathbf{u}} \mathbf{v}|_{C^{\alpha}(\overline{\Omega})} + |\mathbf{v}|_{C^{2,\alpha}(\partial\Omega)} \big).$$
To see the inequality, we proceed with contradiction. If not, then we could find a sequence $\mathbf{v}_j$ such that $|\mathbf{v}_j|_{C^{2,\alpha}(\overline{\Omega})}=1$ and $|L_{\mathbf{u}} \mathbf{v}_j|_{C^{\alpha}(\overline{\Omega})}\rightarrow 0, |\mathbf{v}_j|_{C^{2,\alpha}(\partial\Omega)} \rightarrow 0$. Hence a subsequence $\mathbf{v}_{j^*}$ converges in $C^{2,\beta}$ to some $\mathbf{v}_\infty \in C^{2,\beta}$ for $\beta\in(0,\alpha)$. Then $L_{\mathbf{u}}\mathbf{v}_\infty=0$ and $\mathbf{v}_\infty|_{\partial\Omega}=0$, which implies $\mathbf{v}_\infty=0$. Thus $|\mathbf{v}_{j^*}|_{C^{2,\beta}(\overline{\Omega})}\rightarrow 0$. By Schauder estimates we would have  
$|\mathbf{v}_{j^*}|_{C^{2,\alpha}(\overline{\Omega})}\leq C(|L_{\mathbf{u}} \mathbf{v}_j|_{C^{\alpha}(\overline{\Omega})} + |\mathbf{v}_j|_{C^{2,\alpha}(\partial\Omega)} +|\mathbf{v}_j|_{C^0(\overline{\Omega})})\rightarrow 0$,  
a contradiction to $|\mathbf{v}_{j^*}|_{C^{2,\alpha}(\overline{\Omega})}=1$.

Thus $D\Phi_{\mathbf{u}}$ is bijective and its inverse is bounded.

\medskip

\textbf{Step 3. Structure of the solution set.}  

Define $\mathcal{M} := \Phi^{-1}(\{0\}\times C^{2,\alpha}(\partial\Omega))$. This is exactly the same as the $\mathcal{M}$ defined in \eqref{eq:lifted equation moduli space definition}. By the implicit function theorem for Banach spaces, $\mathcal{M}$ is a smooth Banach submanifold of $C^{2,\alpha}(\overline{\Omega})$.

\medskip

\textbf{Step 4. Properness of the boundary restriction.}  

We prove that the map $\Phi|_{\mathcal{M}} : \mathcal{M} \to C^{2,\alpha}(\partial\Omega)$ given by boundary restriction is proper. Suppose we have a sequence of pairs $(\mathbf{u}_j,\boldsymbol{\varphi_j})$ solving \eqref{eq:TypeII_lifted_pde}-\eqref{eq:bc_TypeII_lifted_pde}
with $\boldsymbol{\varphi_j} \to \boldsymbol{\varphi_\infty}$ in $C^{2,\alpha}(\partial\Omega)$. By Theorem~\ref{thm:higher_regularity_estimates}, there is a uniform bound $|\mathbf{u}_j|_{C^{2,\alpha}(\overline{\Omega})} \leq C$. Fix $\beta \in (0,\alpha)$. Since the embedding $C^{2,\alpha}(\overline{\Omega}) \hookrightarrow C^{2,\beta}(\overline{\Omega})$ is compact, there exists a subsequence $\mathbf{u}_{j^*}$ converging to some $\mathbf{u}_\infty$ in $C^{2,\beta}(\overline{\Omega})$. 

Denote the linear operators $L'_{\mathbf{u}_{j^*}}, L'_{\mathbf{u}_{\infty}}$ by $$L'_{\mathbf{u}_{j^*}} = \Delta_x + \psi \mathrm{div}_z(e^{\mathbf{u}_{j^*}}\nabla_\mathbf{z})- \tfrac{1}{2}\big(2\psi_\xi + a \psi\big)\mathbf{z}\cdot\nabla_{\mathbf{z}}(e^{\mathbf{u}_{j^*}}\cdot)
   + 2K_{\Sigma}e^{\mathbf{u}_{j^*}},$$
$$L'_{\mathbf{u}_{\infty}} = \Delta_x + \psi\mathrm{div}_\mathbf{z}(e^{\mathbf{u}_{\infty}}\nabla_\mathbf{z})- \tfrac{1}{2}\big(2\psi_\xi + a \psi\big)\mathbf{z}\cdot\nabla_{\mathbf{z}}(e^{\mathbf{u}_{\infty}}\cdot)
   + 2K_{\Sigma}e^{\mathbf{u}_{\infty}}.$$
Since $\mathbf{u}_{j^*}\to \mathbf{u}_\infty$ in $C^{2,\beta}$, the coefficients of $L'_{\mathbf{u}_{j^*}}$, written in non-divergence form, converge in $C^\alpha$ (in fact in $C^{1,\beta}$) to the coefficients of $L'_{\mathbf{u}_\infty}$. Consider $\eta = \mathbf{u}_{j^*}-\mathbf{u}_\infty$. Then $\eta$ satisfies $L'_{\mathbf{u}_\infty} \eta = (L'_{\mathbf{u}_\infty}-L'_{\mathbf{u}_{j^*}}) \mathbf{u}_{j^*}$. By Schauder estimates applied to $\eta$, together with the uniform bound $|\mathbf{u}_{j^*}|_{C^{2,\alpha}(\overline{\Omega})}\leq C$, we conclude $\eta \to 0$ in $C^{2,\alpha}(\overline{\Omega})$. Therefore $\mathbf{u}_{j^*}\to \mathbf{u}_\infty$ in $C^{2,\alpha}(\overline{\Omega})$, which proves properness.

\medskip

\textbf{Step 5. Global uniqueness with zero boundary data.} 

By maximum principle, we know that any $\mathbf{u}\in \mathcal{M}$ with $\Phi(\mathbf{u})=0$ implies $\mathbf{u}=0$.

\medskip

\textbf{Final conclusion.}  

The map $\Phi|_{\mathcal{M}} : \mathcal{M} \to C^{2,\alpha}(\partial\Omega)$ is proper and a local diffeomorphism. Since $\mathcal{M},C^{2,\alpha}(\partial\Omega)$ are metric spaces, and $\mathcal{M}$ is non-empty, we conclude that $\Phi|_{\mathcal{M}}$ is a finite-sheeted covering map. Because $(\Phi|_{\mathcal{M}} )^{-1}(0)=\{0\}$, we conclude $\Phi|_{\mathcal{M}} : \mathcal{M} \to C^{2,\alpha}(\partial\Omega)$ is a diffeomorphism.
\end{proof}

\begin{lemma}\label{lem:no cokernel for linearized operator when genus greater than 1}
    Let $L_{\mathbf{u}}^*$ denote the formal adjoint of the operator $L_{\mathbf{u}}$ defined in \eqref{eq:linearized operator general case}. Then the boundary value problem for $\mathbf{w}\in C^2(\overline{\Omega})$
    \begin{equation*}
        L_{\mathbf{u}}^* \mathbf{w} = 0 \quad \text{in } \Omega, 
        \qquad \mathbf{w} = 0 \quad \text{on } \partial\Omega
    \end{equation*}
    only has the trivial solution $\mathbf{w}\equiv 0$.
\end{lemma}

\begin{proof}
   
Recall 
$$
L_{\mathbf{u}}\mathbf{v}
= \Delta_x \mathbf{v} 
   + \psi \Delta_{\mathbf{z}}\!\big(e^{\mathbf{u}}\mathbf{v}\big)
   - \tfrac{1}{2}\big(2\psi_\xi + a \psi\big)\mathbf{z}\cdot\nabla_{\mathbf{z}}(e^{\mathbf{u}}\mathbf{v})
   + 2K_{\Sigma}e^{\mathbf{u}}\mathbf{v}.
$$
Its formal adjoint operator can be computed as follows:
\begin{equation}
\begin{aligned}
L_{\mathbf{u}}^*\mathbf{w}
&= \Delta_x\mathbf{w}
   + e^{\mathbf{u}}\Delta_{\mathbf{z}}(\psi\mathbf{w})
   + \nabla_{z_i}\!\Big(\tfrac{1}{2}(2\psi_\xi+a \psi)\,z_i\mathbf{w}\Big)e^{\mathbf{u}}
   + 2K_{\Sigma}e^{\mathbf{u}}\mathbf{w} \\
&= \Delta_x\mathbf{w}
   + e^{\mathbf{u}}\psi\Delta_{\mathbf{z}}\mathbf{w}
   + 2e^{\mathbf{u}}\nabla_{\mathbf{z}}\psi \cdot \nabla_{\mathbf{z}}\mathbf{w}
   + \tfrac{1}{2}e^{\mathbf{u}}(2\psi_\xi+a \psi)\,\mathbf{z}\cdot\nabla_{\mathbf{z}}\mathbf{w}  \\
&\quad + e^{\mathbf{u}}\Big(2K_{\Sigma}+\Delta_{\mathbf{z}}\psi
   +2(2\psi_\xi+a \psi)
   +\tfrac{1}{2}(\mathbf{z}\cdot\nabla_{\mathbf{z}})(2\psi_\xi+a \psi)\Big)\mathbf{w}.
\end{aligned}
\end{equation}
Using the following 
$$
-\tfrac{1}{2}\,\mathbf{z}\cdot\nabla_{\mathbf{z}}(a\psi)=(\tfrac{1}{2}-\xi)(a\psi)_\xi, \quad  \Delta_{\mathbf{z}}\psi=(\tfrac{1}{2}-\xi)\psi_{\xi\xi}-2\psi_{\xi}, \quad \nabla_{\mathbf{z}} \psi\cdot\nabla_{\mathbf{z}} \mathbf{w}=-\frac{\psi_{\xi}}{2}\mathbf{z}\cdot\nabla_{\mathbf{z}}\mathbf{w},
$$  
and the ODE for $e^{\overline{{w}}}=(\tfrac{1}{2}-\xi)\psi$  
\begin{equation}
\begin{aligned}
\big((\tfrac{1}{2}-\xi)\psi\big)_{\xi\xi}
+ a\big((\tfrac{1}{2}-\xi)\psi\big)_{\xi}
+ b(\tfrac{1}{2}-\xi)\psi=2K_{\Sigma},
\end{aligned}
\end{equation}
the formal adjoint simplifies as  
\begin{equation}
\begin{aligned}
L_{\mathbf{u}}^*\mathbf{w}
&= \Delta_x\mathbf{w}
   + e^{\mathbf{u}}\psi\Delta_{\mathbf{z}}\mathbf{w}
   + \tfrac{a\psi}{2}e^{\mathbf{u}}\mathbf{z}\cdot\nabla_{\mathbf{z}}\mathbf{w}+ e^{\mathbf{u}}\psi\big(a+(\tfrac{1}{2}-\xi)(b-a_{\xi})\big)\mathbf{w}.
\end{aligned}
\end{equation}

For any function $q(\xi)\in C^2([0,\tfrac{1}{2}])$, one can lift it to a $C^2$ function $q(\mathbf{z})$ on the closed ball $|\mathbf{z}|\leq \sqrt{2}$ by setting $\xi=\tfrac{1}{2}-\tfrac{1}{4}|\mathbf{z}|^2$. Then we introduce the conformally related operator $L_{\mathbf{u},q}^*\mathbf{w}:=e^{-q}L_{\mathbf{u}}^*(e^q\mathbf{w})$, where the choice of ${q}$ will be determined later. A computation shows
$$
L_{\mathbf u,q}^*\mathbf w
=\Delta_x\mathbf w
+e^{\mathbf u}\!\left[
\,\psi\,\Delta_{\mathbf z}\mathbf w
+\Big(\tfrac{a \psi}{2}\,\mathbf z+2\psi\,\nabla_{\mathbf z}q\Big)\!\cdot\!\nabla_{\mathbf z}\mathbf w
+\mathcal C\,\mathbf w
\right],
$$
with
$$
\begin{aligned}
\mathcal{C}
&= \psi\big(\Delta_{\mathbf z}q+|\nabla_{\mathbf z}q|^2\big)
+ \tfrac{a\psi}{2}\,\mathbf z\cdot\nabla_{\mathbf z}q
+ \big(2\psi_\xi+2a\psi\big) -\Big(\tfrac12-\xi\Big)\big(\psi_{\xi\xi}+a_\xi \psi+a\psi_\xi\big)
 \\
&= 
\psi\Big(e^{-{q}}\big((\tfrac12-\xi)e^{q}\big)_{\xi\xi}
- \big(\tfrac12-\xi\big)a{q}_\xi
+ a
+ \big(\tfrac12-\xi\big)\big(b-a_\xi\big)\Big).
\end{aligned}
$$
Denote  
$$\phi=\left(\tfrac12-\xi\right)e^q,$$  
then one can simply the expression for $\mathcal{C}$ as 
$$\mathcal{C}=\psi e^{-{q}}\big(\phi_{\xi\xi}-a\phi_{\xi}+(b-a_{\xi})\phi\big).$$  
Recall that $e^{\overline{w}}=(\frac{1}{2}-\xi)\psi$ solves the ODE  
$$(e^{\overline{w}})_{\xi\xi}+a (e^{\overline{w}})_{\xi}+b e^{\overline{w}}=2K_{\Sigma}.$$  
In the following, we take  
$${q}=\log \psi+\int_0^\xi a(\tau)\,d\tau,$$  
which belongs to $C^2([0,\tfrac{1}{2}])$, hence
$$\phi=e^{\int_0^\xi a(\tau)\,d\tau}\,e^{\overline{w}}.$$
Now,  
$$\mathcal{C}=\psi e^{-q}\cdot 2K_{\Sigma}e^{\int_0^\xi a(\tau)\,d\tau}=2K_{\Sigma}\leq 0.$$  
As a consequence, the maximum principle goes through for the operator $L_{\mathbf{u},q}^*$. Therefore, the Dirichlet boundary value problem  
$$L_{\mathbf{u},q}^*\mathbf{w}=0,\qquad \mathbf{w}|_{\partial\Omega}=0$$  
admits only the trivial solution. Since $L_{\mathbf{u},q}^*=e^{-q}L_{\mathbf{u}}^*(e^q\cdot)$, we obtain the desired conclusion.
\end{proof}

\subsection{Regularity of the Einstein metric}\label{subsec:regularity of the Einstein metric}
We have shown that the Dirichlet boundary value problem Question \ref{ques:pde fill in} has unique solution $w$. In this subsection, although our statements are formulated under the assumption $\ttg \geq 1$, this condition is not actually needed for the arguments below. We first prove that
\begin{proposition}
    Any regular conformally K\"ahler Poincar\'e-Einstein $(M,h,g)$ with $M$ diffeomorphic to a complex line bundle over a Riemann surface with $\ttg\geq1$ arises from a solution $w$ to \eqref{eq:pde fill in} with admissible $(deg,\chi,k,\ma,\mpp)$, where $w-\log(\frac{1}{2}-\xi)\in C^\infty([0,\frac{1}{2}]\times\Sigma)$. 
\end{proposition}
\begin{proof}
    As we mentioned before, our calculation for decoupled solutions hold in the averaged sense for general regular conformally K\"ahler Poincar\'e-Einstein, hence we know the associated tuple $(deg,\chi,k,\ma,\mpp)$ is admissible.

    Since the fixed point set of the $\bS^1$-action induced by $\mathcal{K}$ only consists of the bolt $\Xi\simeq\Sigma$ along which the moment map $\xi=\frac{1}{2}$, away from $\Xi$, one can perform K\"ahler reduction for $g$ and write $g$ as \eqref{eq:g final}. The Einstein equation reduces to \eqref{eq:twisted toda final}-\eqref{eq:d eta final}. The solution $w$ is clearly smooth on $\Sigma\times[0,\frac{1}{2})$, but we need to investigate the behavior of $w$ near $\xi=\frac{1}{2}$. As before, fix the gauge by writing the 1-form $\eta=d\theta+\mathcal{X}dx+\mathcal{Y}dy$, with $\theta$ parameterizing the action induced by $\mathcal{K}$ and functions $\mathcal{X},\mathcal{Y}$ determined by \eqref{eq:d eta}. 
    For each point in the bolt $\Xi$, there is the orbit of the induced $\mathbb{C}^*$-action passing through it, with the metric explicitly given by $Wd\xi^2+W^{-1}d\theta^2$. It is direct to calculate that the Gauss curvature of the orbit is explicitly given by $-\frac{1}{2}(W^{-1})_{\xi\xi}$. Hence, since the Gauss curvature of each orbits is smooth as a function on $M$, we necessarily have that $(W^{-1})_{\xi\xi}$ is bounded. Recalling that $W^{-1}$ vanishes at $\xi=\tfrac{1}{2}$, the Morse–Bott property of $\xi$ further implies  
$
W^{-1} = |\mathcal{K}|_g^2=|\nabla_g\xi|^2_{g} \sim A_0^{-1}\Bigl(\tfrac{1}{2}-\xi\Bigr)
$ as $\xi\rightarrow\frac{1}{2}-$
for some positive function $A_0$ on the bolt $\Xi$. As before, admissibility ensures that there are no cone angles along $\Xi$, which implies $\tfrac{\mpp}{2A_0}=2\pi$, and hence $A_0$ is a constant. Thus
$
W - A_0\Bigl(\tfrac{1}{2}-\xi\Bigr)^{-1}
$
is a continuous function satisfying  
$$
\left|\nabla_{\Sigma}^{\,t}\!\left(W - A_0\Bigl(\tfrac{1}{2}-\xi\Bigr)^{-1}\right)\right| < C,
$$
for any $t$. Together with \eqref{eq:W final}, it is now direct to see that we at least have
        $$|\partial_\xi^s\nabla_{\Sigma}^t\bigl(w-\log(\frac{1}{2}-\xi)\bigl)|<C$$
    for $s=0,1$ and any $t$, which in particular implies that $u:=w-\log\!\left(\tfrac{1}{2}-\xi\right)$ extends to a Lipschitz function on $[0,\tfrac{1}{2}]\times\Sigma$. Using the notations of Theorem \ref{thm:higher_regularity_estimates}, the lifted function $\mathbf{u}$ is a weak $W^{1,2}$ solution of \eqref{eq:lifted equation for regularity} in $\Omega$. By interior elliptic regularity, one has $\mathbf{u}\in C^{\infty}(\overline{\Omega'})$, and consequently $u\in C^{\infty}([\tfrac{1}{4},\tfrac{1}{2}]\times\Sigma)$ by Lemma \ref{lem:half regularity}.
\end{proof}

Conversely, we suppose that $w$ is a solution with admissible $(deg,\chi,k,\ma,\mpp)$. Under the additional assumption $W>0$, which will be justified in next section, we prove
\begin{proposition}\label{prop:smooth Einstein metric}
    Given admissible $(deg,\chi,k,\ma,\mpp)$, any solution $w$ to \eqref{eq:pde fill in} with smooth boundary data $\varphi$ and $W>0$ gives rise to a regular conformally K\"ahler Poincar\'e-Einstein $(M,h,g)$ through \eqref{eq:g final}-\eqref{eq:d eta final}, with $M$ diffeomorphic to the complex line bundle over $\Sigma$ of degree $deg$. The K\"ahler metric $g$ on $M$ is the completion of \eqref{eq:g final}.
\end{proposition}
\begin{proof}
    As we did in Section \ref{subsec:conclusion}, one only needs to find a coordinate system around the bolt, in which the Einstein metric $h=\xi^{-2}g$ is Lipschitz. Consider the coordinate change $\tau:=\sqrt{\mpp/\pi}(\frac{1}{2}-\xi)^{\frac{1}{2}}$ as before, and introduce $x_1:=\tau\cos\theta,x_2:=\tau\sin\theta,x_3:=x,x_4:=y$, where $x+iy$ is a local holomorphic coordinate of $\Sigma$. From the regularity Theorem \ref{thm:higher_regularity_estimates}, it is direct to check that the Einstein metric $h$ under $(x_1,x_2,x_3,x_4)$ is Lipschitz. Hence the smoothness of the Einstein metric follows.    
\end{proof}

\subsection{\texorpdfstring{Positivity of $W$}{Positivity of W}}
\label{subsec:positivity of W}

Finally we prove for any solution $w$ to \eqref{eq:pde fill in}, $W>0$. Based on the regularity that $w-\log(\frac{1}{2}-\xi)\in C^\infty$ and $W\equiv1$ when $\xi=0$, for any solution $w$, we do know that
$$
W e^{w}
= \frac{12 e^{w} - 6\xi (e^{w})_{\xi}}{12 + \xi^{3}/k^{3}}
= \frac{e^{u}\,(12 - 6\xi (e^{\overline{w}})_{\xi}) - 6\xi e^{\overline{w}} (e^{u})_{\xi}}{12 + \xi^{3}/k^{3}}
$$
is positive in neighborhoods of $\xi=0$ and $\xi=\tfrac{1}{2}$.

\begin{lemma}
    For any solution $w$ to \eqref{eq:pde fill in} with boundary value $\varphi$, $W=\frac{12-6\xi \partial_\xi w}{12+\xi^3/k^3}$ is positive on $[0,\frac{1}{2})\times\Sigma$.
\end{lemma}
\begin{proof}
    Introduce the space $C^{2,\alpha}(\Sigma,\ma)=\{\varphi\in C^{2,\alpha}(\Sigma)\mid\int_{\Sigma}e^\varphi d\mathrm{vol}_{\Sigma}=\ma\}$ and define $E\subset C^{2,\alpha}(\Sigma,\ma)$ to be the subset of $\varphi$, where the unique solution $w$ with boundary value $\varphi$ (using the same $k$) has positive $W$. The set $E$ is non-empty as it contains the boundary value of the decoupled solution associated to $(deg,\chi,k,\ma,\mpp)$. From the existence and uniqueness of solutions to \eqref{eq:pde fill in}, Theorem \ref{thm:TypeII_lifted_solution_moduli_space}, we know $E$ is open in $C^{2,\alpha}(\Sigma,\ma)$.
    
    We now show that $E$ is closed in $C^{2,\alpha}(\Sigma,\mathfrak{a})$. Suppose $\varphi_j \to \varphi_\infty$ in $C^{2,\alpha}(\Sigma,\ma)$ and $u_j:=w_j-\overline{w} \in C^{2,\alpha}([0,\tfrac{1}{2}]\times\Sigma)$ is the solution with boundary value $\varphi_j\in E$. Passing to a subsequence, we may assume $u_{j^*} \to u_\infty$ in $C^{2,\beta}([0,\tfrac{1}{2}]\times\Sigma)$ for some $0<\beta<\alpha$. By linear elliptic regularity, we actually have $u_{j^*} \to u_\infty$ in $C^{2,\alpha}([0,\tfrac{1}{2}]\times\Sigma)$ (see the last paragraph of Step 4 of Theorem \ref{thm:TypeII_lifted_solution_moduli_space}). For the limit $w_\infty$ there is the associated $W_\infty$, which satisfies $W_\infty=1$ when $\xi=0$ and blows up to positive infinity at $\xi=\frac{1}{2}$ because $u_\infty=w_\infty-\overline{w}\in C^{2,\alpha}$. As $W_{j^*}>0$ for each $w_{j^*}$, we have $W_\infty\geq0$.
    
    Since $|u_{j^*}|_{C^{2,\alpha}([0,\frac{1}{2}]\times\Sigma)} \leq C$, there exists $\xi_1 \in (0,\tfrac{1}{2})$ such that $W_{j^*} = \frac{12 - 6\xi \partial_\xi w_{j^*}}{12+\xi^3/k^3} > \tfrac{1}{2}$ in $[0,\xi_1)$ for any $j^*$. Because of $\varphi_{j^*}\in E$ and Proposition \ref{prop:smooth Einstein metric}, it follows $w_{j^*}$ already induces smooth Poincar\'e-Einstein $(M,h_{j^*})$. Fix $\xi_0 \in (0,\xi_1)$ and consider the region $\{\xi > \xi_0\}$ in the Poincar\'e--Einstein manifold $(M,h_{j^*})$. For the Killing field $\mathcal{K}_{j^*}$, we have $|\mathcal{K}_{j^*}|_{h_{j^*}}^2=\xi^{-2}W_{j^*}^{-1}$, which is subharmonic by the Bochner formula $\frac{1}{2}\Delta_{h_{j^*}}|\mathcal{K}_{j^*}|_{h_{j^*}}^2=|\nabla_{h_{j^*}} \mathcal{K}_{j^*}|^2_{h_{j^*}}-\text{Ric}_{h_{j^*}}(\mathcal{K}_{j^*},\mathcal{K}_{j^*})$. Thus
    $$
    \sup_{\xi > \xi_0} |\mathcal{K}_{j^*}|_{h_{j^*}}^2 = \sup_{\xi = \xi_0} |\mathcal{K}_{j^*}|_{h_{j^*}}^2.
    $$
    Thus $W_{j^*} \geq \tfrac{\xi_0^2}{2\xi^2}$ for any $\xi > \xi_0$. Letting $j^* \to \infty$ gives $W_\infty \geq \tfrac{\xi_0^2}{2\xi^2}$ over $(\xi_0,\frac{1}{2})$. Since we already have $W_{j^*}>\frac{1}{2}$ over $[0,\xi_1)$, it follows $W_\infty > 0$ on $[0,\tfrac{1}{2})$. Therefore, $E$ is closed. 
\end{proof}

Finally, notice that in the Type II case, different 2-dimensional infinity $(\Sigma,g^\natural)$ and admissible $(deg,\chi,k,\ma,\mpp)$ correspond to non-isometric PEs. This is because for a Type II PE $(M,h)$, only constant scalings of $g$ are K\"ahler among all the conformal changes of $h$. Therefore, the data $(\Sigma,g^\natural)$ and $(deg,\chi,k,\ma,\mpp)$ are canonically related to $(M,h)$.

\Addresses

\end{document}